\documentclass[12pt]{article}

\usepackage{graphicx,amsmath,amssymb,tikz,caption,authblk}

\usepackage[pdftex,pagebackref=false, hidelinks]{hyperref} 

\usepackage[normalem]{ulem}
\usetikzlibrary{calc}

\usepackage[margin = 1in]{geometry}
\usepackage{amsthm,amsmath,amssymb,enumerate,mathtools,comment,graphicx,tikz}
\usetikzlibrary{decorations.pathmorphing}

\newtheorem{theorem}{Theorem}[section]
\newtheorem{claim}{}[theorem]
\newtheorem{lemma}[theorem]{Lemma}

\newtheorem{proposition}[theorem]{Proposition}
\newtheorem{corollary}[theorem]{Corollary}

\theoremstyle{definition}
\newtheorem{definition}[theorem]{Definition}

\theoremstyle{plain} 
\newtheorem{subtheorem}{}[theorem]

\newcommand{\cB}{\mathcal{B}}
\newcommand{\cC}{\mathcal{C}}

\newcommand{\cF}{\mathcal{F}}

\newcommand{\cL}{\mathcal{L}}

\DeclareMathOperator{\si}{si}

\DeclareMathOperator{\cl}{cl}

\newcommand{\del}{\!\setminus\!}

\newcommand{\qbg}{\ensuremath{(G,\cB,\cL,\cF)}}
\newcommand{\qbgp}{\ensuremath{(G',\cB',\cL',\cF')}}
\usetikzlibrary{arrows.meta,decorations.markings}
\tikzset{
    midarrow/.style={
        postaction={
            decorate,
            decoration={
                markings,
                mark=at position 0.5 with {\arrow{stealth}}
            }
        }
    }
}

\newdimen\R
\title{The unbreakable quasi-graphic matroids}

\newcommand*\samethanks[1][\value{footnote}]{\footnotemark[#1]}

\author[1]{Sayantani Bhattacharya\thanks{The authors were supported in part by NSF grant DMS-2452015.}}
\author[1]{John David Clifton\samethanks}
\author[1]{Zach Walsh\samethanks}

\affil[1]{\small Auburn University, Department of Mathematics and Statistics, Auburn, AL, U.S.A.}

\date{\today}

\begin{document}

\maketitle

\begin{abstract} 
A matroid $M$ is unbreakable if it is connected and $M/F$ is connected for every flat $F$ of $M$.
Oxley and Pfeil characterized the unbreakable graphic matroids, and Fife, Mayhew, Oxley, and Semple characterized the graphs underlying $3$-connected unbreakable frame matroids.
We extend the latter result by giving a complete characterization of the $3$-connected unbreakable quasi-graphic matroids.
As a special case we obtain a characterization of the $3$-connected lifted-graphic matroids.
\end{abstract}

\section{Introduction}\label{sec: introduction}

A matroid $M$ is \emph{unbreakable} if it is connected, and $M/F$ is connected for every flat $F$ of $M$.
These matroids were first defined by Oxley and Pfeil \cite{Oxley-Pfeil2022}, and have since accrued a number of interesting equivalent characterizations.
For example, unbreakable matroids can be characterized by forbidden parallel minors \cite[Theorem 1.1]{Oxley-Pfeil2022}, forbidding skew circuits in the dual matroid \cite[Theorem 1.1]{Oxley-Pfeil2022}, or a symmetric circuit exchange axiom in the dual matroid \cite[Theorem 1.1]{Cho-Oxley-Wang-2026}, and unbreakable binary matroids can be characterized by a circuit-difference property \cite[Lemma 1.2]{Drummond-Fife-Grace-Oxley-2020}.
Moreover, unbreakability may be a useful property for inductive arguments involving matroid minors as it generalizes the notion of \emph{roundness}, which played a crucial role in \cite{Geelen-Nelson-2010} and \cite{Nelson-2013}, for example.

In light of this list of properties of unbreakable matroids, it may be useful to know whether or not a given matroid is unbreakable.
The first result in this direction is due to Oxley and Pfeil \cite{Oxley-Pfeil2022}, who characterized the unbreakable regular matroids. 
The following corollary gives a  characterization of the unbreakable graphic matroids (writing $\si(G)$ for the simplification of a graph $G$).

\begin{theorem}[{\cite[Theorem 1.2]{Oxley-Pfeil2022}}] \label{thm: the unbreakable graphic matroids}
If $G$ is a graph with no isolated vertices, then the graphic matroid $M(G)$ is unbreakable if and only if $G$ is loopless and $\si(G)$ is a cycle or a complete graph.
\end{theorem}

Theorem \ref{thm: the unbreakable graphic matroids} was extended by Fife, Mayhew, Oxley, and Semple \cite{Fife-Mayhew-Oxley-Semple2020} to the class of $3$-connected \emph{frame matroids}, which is an important class of matroids that contains all graphic matroids, signed-graphic matroids, and bicircular matroids.
A frame matroid can be defined from a \emph{biased graph}, which is a pair $(G, \cB)$ where $G$ is a graph and $\cB$ is a collection of cycles of $G$ (called \emph{balanced} cycles) that does not contain exactly two cycles from any theta subgraph of $G$.
The following is the main result of \cite{Fife-Mayhew-Oxley-Semple2020}; the lower bound on the number of vertices is sharp.

\begin{theorem}[{\cite[Theorem 1.2]{Fife-Mayhew-Oxley-Semple2020}}] \label{thm: the unbreakable frame matroids}
Let $M$ be a $3$-connected unbreakable frame matroid on a graph $G$ with $|V(G)| > 6$ and no isolated vertices.
Then $\si(G)$ can be obtained from a complete graph by deleting the edges of a path of length at most two.
\end{theorem}

We will extend Theorem \ref{thm: the unbreakable frame matroids} to the class of \emph{quasi-graphic matroids}, which is a broad class of matroids derived from graphs that properly contains the class of frame matroids.
A connected quasi-graphic matroid can be defined from a tuple $(G, \cB, \cL, \cF)$ where $(G, \cB)$ is a biased graph and $(\cB,\cL, \cF)$ is a partition of the cycles of $G$ so that every cycle in $\cL$ shares a vertex with every cycle in $\cF$.
The triple $(\cB, \cL, \cF)$ is a \emph{proper tripartition} of the cycles of $G$.
When $\cL = \varnothing$ the quasi-graphic matroid of $(G, \cB, \cL, \cF)$ is simply the frame matroid of $(G, \cB)$, and when $\cF = \varnothing$ the quasi-graphic matroid of $(G, \cB, \cL, \cF)$ is the \emph{lifted-graphic} matroid of $(G, \cB)$, which is another fundamental matroid that can be defined from a biased graph.
We will give more details about these matroids in Section \ref{sec: background}.
The following is our main result, and it implies Theorem \ref{thm: the unbreakable frame matroids}.

\begin{theorem} \label{thm: main}
Let $M$ be a $3$-connected unbreakable quasi-graphic matroid on a graph $G$ so that $|V(G)| > 6$ and each component has at least two vertices.
Then either $\si(G)$ is a cycle and $M$ is lifted-graphic, or $\si(G)$ can be obtained from a complete graph by deleting the edges of a path of length at most two.
\end{theorem}

The condition that $|V(G)| > 6$ is necessary as there are several examples of a $3$-connected unbreakable frame matroid on a graph $G$ so that $|V(G)| = 6$ and $\si(G)$ is not a cycle or the complement of a path with at most two edges, as shown in Figure \ref{figure: 6-vertex graphs}.
Since adding an isolated vertex or a set of loops in $\cL$ at an isolated vertex does not change the quasi-graphic matroid, the condition that each component has at least two vertices is also necessary. Restricting to lifted-graphic matroids, we obtain the following new result as a corollary.

\begin{corollary} \label{thm: main corollary for lifted-graphic matroids}
Let $M$ be a $3$-connected unbreakable lifted-graphic matroid on a graph $G$ so that $|V(G)| > 6$ and each component has at least two vertices.
Then $\si(G)$ is a cycle, or $\si(G)$ can be obtained from a complete graph by deleting the edges of a path of length at most two.    
\end{corollary}

Fife, Mayhew, Oxley, and Semple also proved a partial converse to Theorem \ref{thm: the unbreakable frame matroids} by showing that if $G$ can be obtained from a complete graph by deleting the edges of a path of length at most two, then there exists a $3$-connected unbreakable frame matroid $M$ on $G$ \cite[Theorem 1.5]{Fife-Mayhew-Oxley-Semple2020}.
We significantly generalize this result by giving a complete characterization of the $3$-connected unbreakable quasi-graphic matroids on a graph $G$ with $|V(G)| > 6$ and $\si(G)$ a cycle or the complement of a path with length at most two.
We separately consider the cases when $\si(G)$ is a cycle and when $\si(G)$ is nearly complete, and in each case we also characterize which quasi-graphic matroids on $G$ are $3$-connected.
The following result shows that when $\si(G)$ is a cycle and $M$ is both lifted-graphic and $3$-connected, $M$ is always unbreakable.

\begin{theorem} \label{thm: main converse for cycles}
Let $M$ be a lifted-graphic matroid on a graph $G$ so that $\si(G)$ is a cycle and $|V(G)| > 6$.
Then $M$ is $3$-connected if and only if $M$ is simple and at most one edge of $G$ is not in a $2$-cycle, and if $M$ is $3$-connected, then $M$ is unbreakable.
\end{theorem}
 
When $\si(G)$ is the complement of a path with at most two edges,
we show that $M$ is unbreakable unless $G$ has a highly structured \emph{balancing set}, which is a set $X$ of edges so that $G$ has an unbalanced cycle but $G \del X$ does not.

The exact statement (Theorem \ref{thm: structure of breakable complete quasi-graphic matroids}) is somewhat technical, so for now we will only give an informal statement. 
However, it is straightforward to determine when $M$ is $3$-connected.

\begin{theorem} \label{thm: main converse for nearly complete graphs}
Let $M$ be a quasi-graphic matroid on a graph $G$ so that $\si(G)$ is the complement of a path with length at most two and $|V(G)| > 6$.
Then $M$ is $3$-connected if and only if $M$ is simple and $G$ has no balancing set $X$ with $r_M(X) \le 2$, and $M$ is not unbreakable if and only if $G$ has a highly structured balancing set.
\end{theorem}

Theorems \ref{thm: main}, \ref{thm: main converse for cycles}, and \ref{thm: main converse for nearly complete graphs} together give a complete characterization of the $3$-connected unbreakable quasi-graphic matroids on graphs with more than six vertices.
This leads to the following question: what about unbreakable quasi-graphic matroids that are not $3$-connected? 
If $M$ is obtained from a $3$-connected unbreakable matroid by making a parallel copy of an element, then $M$ is not $3$-connected but is unbreakable, so it suffices to consider simple matroids.
We give a complete answer to the above question in Section \ref{sec: future work} by considering two cases separately.
We say that $\cL$ (respectively $\cF$) is \emph{degenerate} if it contains no vertex-disjoint cycles, and is otherwise \emph{non-degenerate.} If $\cL$ (respectively $\cF$) is degenerate, then $M\qbg$ is the frame matroid (respectively lifted-graphic matroid) of $(G,\cB)$.
We will show that, for simple quasi-graphic matroids $M=M\qbg$, the structure imposed by having non-degenerate collections $\cL$ and $\cF$ is incompatible with the structure imposed by being unbreakable when $M$ is not $3$-connected.

\begin{theorem}\label{thm: non-degen unbreakability}
   Let $M=M\qbg$ be a simple connected quasi-graphic matroid such that neither $\cL$ nor $\cF$ is degenerate. If $M$ is not 3-connected, then $M$ is not unbreakable.
\end{theorem}

In light of the above results, to obtain a complete characterization of the unbreakable quasi-graphic matroids it suffices to consider frame matroids or lifted-graphic matroids that are not $3$-connected. To this end, we utilize additional results of \cite{Oxley-Pfeil2022}, along with Theorem \ref{thm: main}, to articulate a 2-sum decomposition of the connected but not 3-connected unbreakable frame and lifted-graphic matroids. 
Here a \emph{free element} of a matroid is an element that is in a circuit and is only in spanning circuits. 
\begin{theorem}\label{thm: connected not 3-connected}
    Let $M$ be a simple frame or lifted-graphic matroid that is connected but not 3-connected. Then $M$ is unbreakable if and only if $M$ is obtained via 2-sums along free elements with a sequence of either unbreakable frame matroids or unbreakable lifted-graphic matroids such that:
    \begin{enumerate}[$(i)$]
        \item If $M$ is a frame matroid, the rank of each summand is at most 6.
        \item If $M$ is lifted-graphic, each summand with rank greater than 6 is on a graph which simplifies to a cycle.
        \item If any summand is also graphic, it has a graph which is a cycle on exactly 3 vertices.
    \end{enumerate}
\end{theorem}

The rest of this paper is structured as follows.
We will conclude this section by fixing some notation for graphs.
In Section \ref{sec: background} we will give some background on biased graphs, frame matroids, lifted-graphic matroids, and quasi-graphic matroids.
In Section \ref{sec: preliminaries} we will state some basic lemmas, many of which appear in \cite{Fife-Mayhew-Oxley-Semple2020}, about biased graphs and unbreakable matroids.
In Section \ref{sec: when G is not 3-connected} we will prove Theorem \ref{thm: main} in the case that $G$ is not $3$-connected, and in Section \ref{sec: when G is 3-connected} we will prove Theorem \ref{thm: main} in the case that $G$ is $3$-connected.
Both of these proofs will follow \cite{Fife-Mayhew-Oxley-Semple2020} very closely, generalizing each step from frame matroids to quasi-graphic matroids.
We will prove Theorems \ref{thm: main converse for cycles} and \ref{thm: main converse for nearly complete graphs} in Section \ref{sec: characterization}, and conclude in Section \ref{sec: future work} by proving Theorems \ref{thm: non-degen unbreakability} and \ref{thm: connected not 3-connected}.
We will assume familiarity with matroid theory, and use notation from \cite{Oxley2011} unless stated otherwise.

\bigskip
\noindent 
\textbf{Notation.} 
Let $G = (V, E)$ be a graph.
For a set $W$ of vertices of $G$ we write $G[W]$ for the subgraph of $G$ induced by $W$.
For a set $X$ of edges of $G$ we write $G[X]$ for the graph with edge set $X$ and vertex set consisting of all ends of edges in $X$, and $c(X)$ for the number of components of $G[X]$.
For a vertex $v$ we write $E_v$ for the set of edges incident with $v$.
The \emph{simplification} of $G$, denoted $\si(G)$, is any graph obtained from $G$ by deleting all loops and all but one element from each maximal set of parallel edges.
For disjoint sets $V_1$ and $V_2$ of vertices of $G$ we write $E[V_1, V_2]$ for the set of edges of $G$ with one end in $V_1$ and one end in $V_2$.

\section{Background} \label{sec: background}

In this section we will give some background information on biased graphs, frame matroids, lifted-graphic matroids, and quasi-graphic matroids.

\subsection{Biased graphs and their matroids} \label{sec: biased graphs and their matroids}

A \emph{theta graph} is a graph consisting of two vertices with three internally vertex-disjoint paths between them.
A \emph{biased graph} is a pair $(G, \cB)$ where $G$ is a graph and $\cB$ is a set of cycles of $G$ that satisfies the \emph{theta property}: no theta subgraph contains exactly two cycles in $\cB$.
Biased graphs were originally introduced by Zaslavsky \cite{Zaslavsky1989} as combinatorial objects that store information about certain group-labelings of the edges of a graph.
The cycles in $\cB$ are \emph{balanced} and the cycles not in $\cB$ are \emph{unbalanced}.
More generally, a set $X$ of edges of $G$ is \emph{balanced} if every cycle of $G[X]$ is balanced.

Every biased graph $(G, \cB)$ has two associated matroids on $E(G)$, both originally defined by Zaslavsky \cite{Zaslavsky1991}, in which the circuits are particular types of biased subgraphs.
A \emph{tight handcuff} is a graph consisting of two edge-disjoint cycles with exactly one common vertex, and a \emph{loose handcuff} is a graph consisting of two vertex-disjoint cycles with a minimal path connecting them.
Equivalently, they are subdivisions of the graphs shown in Figure \ref{fig: handcuffs}.
The \emph{frame matroid} of $(G, \cB)$, denoted $F(G, \cB)$, is the matroid on $E(G)$ whose circuits are the balanced cycles, tight handcuffs with both cycles unbalanced, theta graphs with all cycles unbalanced, and loose handcuffs with both cycles unbalanced.
Note that if all cycles are balanced, then $F(G, \cB)$ is the graphic matroid of $G$.
Important special types of frame matroids include signed-graphic matroids, Dowling geometries, and bicircular matroids.

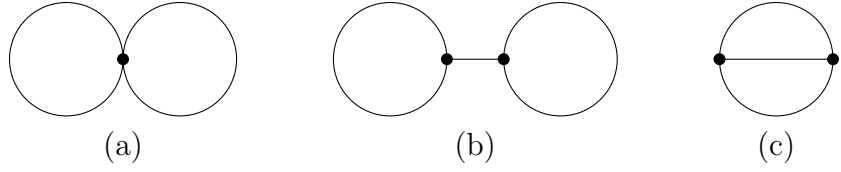
\begin{figure}
    \centering
\begin{tabular}{ccc ccc c}
 \begin{tikzpicture}[scale=0.75]

\begin{scope}
    \draw (0,0) circle (1);
    \draw (2,0) circle (1);
    \fill (1,0) circle (3pt);
\end{scope}

\end{tikzpicture}
     &&&
 \begin{tikzpicture}[scale=0.75]

\begin{scope}
    \draw (0,0) circle (1);
    \draw (3,0) circle (1);
    \fill (1,0) circle (3pt);
    \fill (2,0) circle (3pt);
    \draw (1,0) -- (2,0);
\end{scope}

\end{tikzpicture}

     &&&
 \begin{tikzpicture}[scale=0.75]

\begin{scope}
    \draw (0,0) circle (1);
    \fill (-1,0) circle (3pt);
    \fill (1,0) circle (3pt);
    \draw (-1,0) -- (1,0);
\end{scope}

\end{tikzpicture}

\\
(a) &&& (b) &&& (c)

\end{tabular}
       
    \caption{Tight handcuffs, loose handcuffs, and theta graphs are subdivisions of the graphs (a), (b), and (c), respectively.}
    \label{fig: handcuffs}
\end{figure}

The \emph{lift matroid} of $(G, \cB)$, denoted $L(G, \cB)$, is the matroid on $E(G)$ whose circuits are the balanced cycles, tight handcuffs with both cycles unbalanced, theta graphs with all cycles unbalanced, and unions of two vertex-disjoint unbalanced cycles.
If all cycles are balanced, $L(G, \cB)$ is the graphic matroid of $G$.
A matroid is \emph{lifted-graphic} if it is isomorphic to the lift matroid of a biased graph.
Lifted-graphic matroids are very close to being graphic: a matroid $M$ is lifted-graphic if and only if there is a matroid $N$ with an element $e$ so that $N\del e = M$, and $N/e$ is graphic.
Important special cases of lifted-graphic matroids include even-cycle matroids and spikes.

\subsection{Quasi-graphic matroids} \label{sec: quasi-graphic matroids}

While the classes of frame matroids and lifted-graphic matroids are fundamental, they have the following significant drawbacks: both classes have an infinite list of excluded minors \cite{Chen-Geelen18}, and neither class has a polynomial-time recognition algorithm \cite{Chen-Whittle18}.
In an attempt to bypass these undesirable properties, Geelen, Gerards, and Whittle \cite{Geelen-Gerards-Whittle18} defined a common generalization of frame matroids and lifted-graphic matroids.
A matroid $M$ is \emph{quasi-graphic} if it has a \emph{framework}: that is, a graph $G$ with $(i) \ E(G)=E(M)$, such that $(ii)$ the rank of the edge set of each component of $G$ is at most the size of its vertex set, $(iii)$ for each vertex $v \in V(G)$ the closure in $M$ of $E(G-v)$ does not contains an edge with endpoints $v,w$ with $w \ne v$, and $(iv)$ no circuit of $M$ induces a subgraph in $G$ of more than two components. 

The class of quasi-graphic matroids is closed under direct sums (by taking disjoint unions of frameworks) and minors \cite[Lemma 4.1]{Geelen-Gerards-Whittle18}, and contains the class of frame matroids \cite[Theorem 1.3]{Geelen-Gerards-Whittle18} and the class of lifted-graphic matroids \cite[Theorem 1.2]{Geelen-Gerards-Whittle18}.
Moreover, it is conjectured that the list of excluded minors for the class of quasi-graphic matroids is finite \cite[Conjecture 1.5]{Chen-Geelen18}, and that there is a polynomial-time algorithm for recognition of $3$-connected quasi-graphic matroids \cite[Conjecture 1.7]{Geelen-Gerards-Whittle18}.

\bigskip
\noindent 
\textbf{Quasi-biased graphs.}
It is straightforward to show that if $M$ is a quasi-graphic matroid with framework $G$ then the set $\cB$ of cycles of $G$ that are circuits of $M$ satisfies the theta property \cite[Lemma 3.2]{Geelen-Gerards-Whittle18}, so every quasi-graphic matroid has a unique associated underlying biased graph.
While this biased graph does not uniquely determine $M$, Bowler, Funk, and Slilaty \cite{Bowler-Funk-Slilaty24-corrigendum} show if $G$ or $M$ is connected then $M$ can be recovered with slightly more information about the unbalanced cycles.
If $G$ is a graph and $(\cB, \cL, \cF)$ is a partition of its cycles, then $(\cB, \cL, \cF)$ is a \emph{proper tripartition} if the cycles in $\cB$ satisfy the theta property and every cycle in $\cL$ meets every cycle in $\cF$.
An unbalanced cycle is \emph{frame-type} if it is in $\cF$ and \emph{lift-type} if it is in $\cL$.
Bowler, Funk, and Silaty \cite[Theorem 2.3]{Bowler-Funk-Slilaty20} show that there is a matroid $M(G, \cB, \cL, \cF)$ with ground set $E(G)$ and the following circuits:
\begin{itemize}
    \item cycles in $\cB$,
    \item theta subgraphs with no cycles in $\cB$,
    \item tight handcuffs with neither cycle in $\cB$,
    \item unions of two vertex-disjoint cycles in $\cL$, and
    \item loose handcuffs with both cycles in $\cF$.
\end{itemize}
They prove that the class of matroids that arise in this way is very closely related to the class of quasi-graphic matroids.

\begin{theorem} [{\cite[Theorem 2]{Bowler-Funk-Slilaty24-corrigendum}}] \label{thm: equivalence of frameworks and proper tripartitions}
Let $M$ be a matroid and let $(G, \cB)$ be a biased graph with $E(G) = E(M)$.
Consider the following statements.
\begin{enumerate}[$(1)$]
\item $M$ is quasi-graphic with framework $G$ and $\cB$ is the set of cycles of $G$ that are circuits of $M$.

\item There is a proper tripartition $(\cB, \cL, \cF)$ of the cycles of $G$ such that $M = M(G, \cB, \cL, \cF)$.
\end{enumerate}
Then $(2)$ implies $(1)$, and if $M$ is connected or $G$ is connected, then $(1)$ implies $(2)$.
\end{theorem}

It is straightforward to see that the condition in Theorem \ref{thm: equivalence of frameworks and proper tripartitions} that $M$ or $G$ is connected cannot be removed: the direct sum of a frame matroid and a lifted-graphic matroid is quasi-graphic, but does not always arise via a proper tripartition.

While our main results concern quasi-graphic matroids, it will be more convenient to work with matroids that arise from proper tripartitions.
To this affect, we make the following definition.

\begin{definition}
A \emph{quasi-biased graph} is a tuple $(G, \cB, \cL, \cF)$ where $G$ is a graph and $(\cB, \cL, \cF)$ is a proper tripartition of the cycles of $G$.
\end{definition}

Since we will only be concerned with connected matroids in this paper, we will not lose any generality by considering the quasi-graphic matroids that arise from quasi-biased graphs.
We say that $(G, \cB, \cL, \cF)$ is \emph{connected} if $G$ is connected.
One benefit of working with quasi-biased graphs instead of frameworks is that it is very clear how matroids from quasi-biased graphs relate to frame matroids and lifted-graphic matroids.
Given a quasi-biased graph $(G,\cB,\cL,\cF)$, the collection $\cL$ (respectively $\cF$) is \textit{degenerate} if no two cycles of $\cL$ (respectively $\cF$) are vertex-disjoint, and is otherwise \emph{non-degenerate}. Observe that if $\cL$ is degenerate, then $M(G, \cB, \cL, \cF)$ is equal to $F(G, \cB)$, the frame matroid of $(G, \cB)$. 
Similarly if $\cF$ is degenerate, then $M(G, \cB, \cL, \cF)$ is equal to $L(G, \cB)$, the lifted-graphic matroid of $(G, \cB)$.
In particular, note that if either of $\cL$ or $\cF$ contains a loop, then necessarily the other set is degenerate as all cycles in $\cL$ or $\cF$ must meet this loop.

We comment that if $M$ is a connected quasi-graphic matroid with rank at least two, then there is a quasi-biased graph $(G, \cB, \cL, \cF)$ so that $M = M(G, \cB, \cL, \cF)$ and each component of $G$ has at least two vertices. To see this, note that a component with only one vertex consists of either an isolated vertex or a vertex with incident loops. Clearly deleting isolated vertices does not change the underlying quasi-graphic matroid. Since $M$ is connected it has no balanced loops, and it is straightforward to see that if a vertex is incident with any lift-type loops, these loops may be freely placed on any vertex of $G$ without changing the circuits of the matroid. If any component consists of a vertex incident with a frame-type loop, this contradicts that $M$ is connected since $M$ has rank at least two.

\bigskip
\noindent 
\textbf{Rank and minors.}
We will need to know the independent sets and rank function of matroids associated with quasi-biased graphs.
Bowler, Funk, and Slilaty stated the rank function of such a matroid \cite[Lemma 2.4]{Bowler-Funk-Slilaty20}.

\begin{proposition}[{\cite[Lemma 2.4]{Bowler-Funk-Slilaty20}}]\label{prop: rank of X}
Let $(G, \cB, \cL, \cF)$ be a quasi-biased graph and let $X$ be a subset of $E(G)$. Then the rank of $X$ in $M(G,\cB,\cF,\cL)$ is given by:
\begin{align*}
    r(X)=\begin{cases} \lvert V(G[X])\rvert-b(X) & \textit{if $X$ contains a cycle in $\cF$}\\\\
    \lvert V(G[X])\rvert-c(X)+l(X) &\textit{otherwise}
    \end{cases}  
\end{align*}
where $b(X)$ is the number of balanced components of $G[X]$, $c(X)$ is the number of components of $G[X]$, and $l(X)$ is $1$ if $X$ contains an unbalanced cycle and $0$ otherwise. 
\end{proposition}

In particular, note that if $G$ is connected and has an unbalanced cycle, then the rank of $M(G, \cB, \cF, \cL)$ is equal to $|V(G)|$.
From the rank function we can recover the independent sets.

\begin{proposition}[{\cite[Theorem 5.1]{Bowler-Funk-Slilaty20}}]\label{prop: independent sets of a quasi-graphic matroid}
Let $(G, \cB, \cL, \cF)$ be a quasi-biased graph and $X$ be a subset of $E(G)$. Then $X$ is independent in $M(G,\cB,\cF,\cL)$ if and only if $G[X]$ satisfies one of the following:
\begin{enumerate}[$(i)$]
    \item $G[X]$ is a forest.
    \item $G[X]$ has exactly one cycle, which is in $\cL$.
    \item Each component of $G[X]$ has at most one cycle, each of which is in $\cF$.
\end{enumerate}
\end{proposition}

Crucially for our work, the class of matroids arising from quasi-biased graphs is minor-closed. 
We next describe deletion and contraction for quasi-biased graphs, following \cite{Bowler-Funk-Slilaty20}. 
Let $(G,\cB,\cL,\cF)$ be a quasi-biased graph. 
Let $(G,\cB,\cL,\cF)\del e$ denote the graph $G\del e$ together with the tripartition $(\cB',\cL',\cF')$ of the cycles of $G \del e$ obtained from $(\cB,\cL,\cF)$ by taking
\begin{align*}
    \cB' &= \{C : C \in \cB \ \text{and} \ C \ \text{does not contain} \ e\} \\
    \cL' &= \{C : C \in \cL \ \text{and} \ C \ \text{does not contain} \ e\} \\
    \cF' &= \{C : C \in \cF \ \text{and} \ C \ \text{does not contain} \ e\}.
\end{align*}
It is straightforward to show that $M(G, \cB, \cL, \cF) \del e = M((G, \cB, \cL, \cF) \del e)$ \cite[Theorem 4.5]{Bowler-Funk-Slilaty20}.
If $e$ is not a loop, let $(G,\cB,\cL,\cF)/e$ denote the graph $G/e$ together with the tripartition $(\cB'',\cL'',\cF'')$ of the cycles of $G/e$ obtained from $(\cB,\cL,\cF)$ by taking
\begin{align*}
    \cB'' &= \{C : C \in \cB \ \text{or} \ C \cup e \in \cB\} \\
    \cL'' &= \{C : C \in \cL \ \text{or} \ C \cup e \in \cL\} \\
    \cF'' &= \{C : C \in \cF \ \text{or} \ C \cup e \in \cF\}.
\end{align*}
Again, it is straightforward to show that $M(G, \cB, \cL, \cF)/e = M((G, \cB, \cL, \cF)/e)$ \cite[Theorem 4.5]{Bowler-Funk-Slilaty20}.

We are left to deal with the contraction of a loop $e$ of $G$.
If $e \in \cB$ then $e$ is also a loop of $M(G, \cB, \cL, \cF)$, so we can define $(G, \cB, \cL, \cF)/e$ to be $(G, \cB, \cL, \cF) \del e$.
If $e \notin \cB$, then $M(G,\cB, \cL, \cF)$ is a frame matroid (if $\{e\} \in \cF$) or a lifted-graphic matroid (if $\{e\} \in \cL$) and we consider these cases separately, as in \cite{Funk-Pivotto-Slilaty22}.
If $\{e\} \in \cL$, we define $(G, \cB, \cL, \cF)/e = (G/e, \cC(G/e), \varnothing, \varnothing)$ where $\cC(G/e)$ is the set of all cycles of $G/e$.
In other words, to contract a loop in $\cL$, delete the loop and declare all cycles to be balanced.
As described in \cite[page 9]{Funk-Pivotto-Slilaty22}, $M(G, \cB, \cL, \cF)/e = M((G, \cB, \cL, \cF)/e)$ in this case.
If $\{e\} \in \cF$, define $(G,\cB, \cL, \cF) / e = (G', \cB', \varnothing, \cC(G') - \cB')$ where $G'$ is obtained from $G$ by replacing each non-loop incident with $v$ by a loop at its other end, and $\cB'$ consists of all cycles in $\cB$ disjoint from $v$ and all loops of $G$ incident with $v$.
To use language from \cite{Funk-Pivotto-Slilaty22}, $G'$ is obtained by taking each non-loop incident with $v$ and ``rolling it up" to a loop at its opposite endpoint. 
As described in \cite[page 9]{Funk-Pivotto-Slilaty22}, $M(G, \cB, \cL, \cF)/e = M((G, \cB, \cL, \cF)/e)$ in this case.

It is straightforward to show that contractions from $(G, \cB, \cL, \cF)$ commute and that contractions and deletions commute, so for disjoint sets $X$ and $Y$ of edges of $G$ we will write $(G, \cB, \cL, \cF)\del X/Y$ for the quasi-biased graph obtained from $(G, \cB, \cL, \cF)$ by deleting the edges in $X$ and contracting the edges in $Y$ in any order.
Note that if $Y$ contains no cycles in $\cF$, then the graph of $(G, \cB, \cL, \cF) \del X/Y$ is $G \del X/Y$; we will freely use this fact.

We will need two straightforward properties of contraction for quasi-biased graphs.
The first helps us understand how unbalanced cycles behave with respect to contraction of balanced sets.

\begin{lemma} \label{lem: unbalanced cycles preserved by contraction}
Let $(G, \cB, \cL, \cF)$ be a quasi-biased graph, let $C$ be an unbalanced cycle, and let $F$ be a flat of $M(G, \cB, \cL, \cF)$ so that $C - F \ne \varnothing$ and $F$ contains no unbalanced cycles.
Then $C - F$ contains an unbalanced cycle of $(G, \cB, \cL, \cF)/F$. 

\end{lemma}
\begin{proof}
It suffices to consider the case when $F$ consists of a single edge $f$ which is not an unbalanced loop. If $f$ is an edge of $C$, then $C-f$ is an unbalanced cycle of $(G,\cB,\cL,\cF)/f$. If $f$ is incident with two vertices $\{u,v\}$ of $C$, then in $G/f$, the edge set $C-f$ is a disjoint union of two cycles $C_1$ and $C_2$. At least one such cycle, say $C_1$ must be unbalanced. To see this, suppose otherwise. Then each of the cycles $C_i\cup f$ for $i=1,2$ is balanced. By the theta property, this forces $C$ to be balanced, a contradiction. Lastly consider the case when $f$ is incident with at most one vertex of $C$. If $f$ is a non-loop edge or a balanced loop, then $C$ is an unbalanced cycle of $(G,\cB,\cL,\cF)/f$. 

\end{proof}

We will also need to understand what happens when we contract a flat that contains a frame-type cycle.

\begin{lemma} \label{lem: at least one frame-type loop}
Let $(G, \cB, \cL, \cF)$ be a quasi-biased graph so that $G$ is connected, and let $M = M(G, \cB, \cL, \cF)$.
If $F$ is a flat of $M$ so that $r(M/F) \ge 1$ and $F$ contains a cycle in $\cF$, then $(G, \cB, \cL, \cF)/F$ has a frame-type loop.
\end{lemma}
\begin{proof}
Let $T$ be an edge-maximal forest of $G[F]$ and let $(G', \cB', \cL', \cF') = (G, \cB, \cL, \cF)/\cl(T)$.
Note that $\cl(T)$ is balanced and so $G' = G/\cl(T)$.
Since $G$ is connected this implies that $G'$ is connected.
Also, each element in $F - \cl(T)$ is a loop of $G'$ in $\cF'$.
Let $V'$ be the set of vertices of $G'$ with an incident loop in $F$.
Then the graph of $(G, \cB, \cL, \cF)/F$ is obtained from $G'$ as follows: for each non-loop edge $e$ with exactly one end in $V'$, replace $e$ with a frame-type loop at the end of $e$ that is not in $V'$.
Since $G'$ is connected, there is at least one non-loop edge with exactly one end in $V'$, so $(G, \cB, \cL, \cF)/F$ has at least one frame-type loop.
\end{proof}

\section{Preliminaries} \label{sec: preliminaries}

In this section we will state some preliminary lemmas for quasi-graphic matroids and unbreakable matroids.
We will follow \cite[Section 2]{Fife-Mayhew-Oxley-Semple2020} very closely, with each lemma either taken verbatim or generalized from frame matroids to quasi-graphic matroids.

The first was proved by Oxley and Pfeil \cite{Oxley-Pfeil2022} and holds for general unbreakable matroids.

\begin{lemma}[{\cite[Lemma 2.1]{Fife-Mayhew-Oxley-Semple2020}}]
If $M$ is an unbreakable matroid and $F$ is a flat of $M$, then $M/F$ is unbreakable.
\end{lemma}

We can use the rank function to certify that every induced subgraph has a natural corresponding flat. The following lemma generalizes \cite[Lemma 2.3]{Fife-Mayhew-Oxley-Semple2020}.

\begin{lemma}\label{lem: vertex induced flats}
Let $(G, \cB, \cL, \cF)$ be a quasi-biased graph, $L$ its set of loops in $\cB$, and $L'$ its set of loops in $\cL$. Let $U\subseteq V(G).$ If $E(G[U])$ contains no cycles in $\cL$, then $E(G[U]) \cup L$ is a flat of $M(G,\cB,\cL,\cF)$. Otherwise $E(G[U])\cup L\cup L'$ is a flat of $M(G,\cB,\cL,\cF)$.
\end{lemma}
\begin{proof}
    Let $\cC$ denote the collection of circuits of $M$. We recall the following characterization of closure for a subset $X\subseteq E(M)$ and any matroid $M$: 
    \begin{align*}
        \cl(X)=X\cup\{x:\exists \ C\in\cC \text{ so that }x\in C\subseteq X\cup x\}. 
    \end{align*}
    First suppose $E(G[U])$ contains no cycles in $\cL$. Then, for any edge $e$, if $e\in C\subseteq (E(G[U])\cup e)$ for some circuit $C$, then $C$ is either a balanced cycle, a theta graph with all cycles unbalanced, or a handcuff with both cycles unbalanced. For each of these cases either $e\in E(G[U])$ or $e$ is a balanced loop.

    Now suppose $E(G[U])$ contains a cycle in $\cL$. Again, for any $e\in C\subseteq (E(G[U])\cup e)$, we have that either $e\in E(G[U])$, $e$ is a balanced loop, or $e$ is a loop in $\cL$.
\end{proof}

We will also need the following lemma, which relies on Theorem 1.1.

\begin{lemma} \label{lem: contract a cycle}
Let $\qbg$ be a quasi-biased graph and $M=M\qbg$. If $C$ is an unbalanced cycle, then $M/C$ is either frame or lifted-graphic, depending on whether $C$ belongs to $\cF$ or $\cL$, respectively. Moreover, if $M$ is unbreakable and  $C\in \cL$, then $\si(G/\cl(C))$ is a complete graph or a cycle up to deleting isolated vertices.
\end{lemma}
\begin{proof}
Let $e\in C$. Contracting the edges of $C-e$, we have an unbalanced loop resulting from $e$. Thus in $\qbg/(C-e)$, one of $\cF$ or $\cL$ is degenerate, and $M/(C-e)$ is either a frame matroid or a lifted-graphic matroid. Now suppose in addition that $C\in \cL$. Then $(G, \cB, \cL, \cF)/C$ has no unbalanced cycles, and its graph is $G/C$. Thus $M/\cl(C)$ is graphic and unbreakable. 
So $\si(G/\cl(C))$ is isomorphic to the cycle matroid of a complete graph or a cycle by Theorem \ref{thm: the unbreakable graphic matroids}.  
\end{proof}

A utility of this result is that we can occasionally default to the results of \cite{Fife-Mayhew-Oxley-Semple2020} even when properties outside of biased graphs are not sufficient, so long as our contraction consists of flats containing cycles in $\cF$.

The following three lemmas from \cite{Fife-Mayhew-Oxley-Semple2020} concern biased graphs only (and not their associated matroids) so they apply in the quasi-graphic setting as well.

\begin{lemma}[{\cite[Lemma 2.4]{Fife-Mayhew-Oxley-Semple2020}}] \label{lem: induced shortest unbalanced cycles}
Let $(G, \cB)$ be a biased graph with no balanced loops or $2$-cycles. Let $H$ be a vertex-induced subgraph of $G$ and let $C$ be a shortest unbalanced cycle in $H$.
If $|C| \ge 3$, then $C$ is induced.
\end{lemma}

\begin{lemma}[{\cite[Lemma 2.5]{Fife-Mayhew-Oxley-Semple2020}}] \label{lem: wheel}
    Let $(G, \cB)$ be a biased graph and $C$ be an unbalanced cycle of $G$. If $w$ is in $V(G)-V(C)$ and $w$ is adjacent to at least two vertices of $C$ then there is an unbalanced cycle $C_w$ with $w \in V(C_w)$ and $E(C_w) \subseteq E(C) \cup E_w$.
\end{lemma}

\begin{lemma}[{\cite[Lemma 2.6]{Fife-Mayhew-Oxley-Semple2020}}] \label{lem: vertex adjacent to every vertex of cycle}
  Let $(G,\cB)$ be a biased graph and $C$ be an unbalanced cycle of $G$ with $|C| \geq 3$. If $G$ has a vertex $w$ that is adjacent to each vertex of $V(C) - w$ then there is an unbalanced $3$-cycle $C_w$ with $w \in V(C_w)$ and $E(C_w) \subseteq E(C) \cup E_w$.  
\end{lemma}

The next lemma will help us identify when a quasi-graphic matroid is not unbreakable. It generalizes \cite[Lemma 2.7]{Fife-Mayhew-Oxley-Semple2020}.

\begin{lemma}\label{lem: tree deletion disconnects G}
Let $M=M(G, \cB,\cL,\cF)$ be a quasi-graphic matroid where $G$ is connected.
Let $C$ be an unbalanced cycle and let $T$ be a tree in $G$ such that $V(C) \subseteq V(T)$ and $G - V(T)$ is disconnected. 
Then $M$ is not unbreakable.
\end{lemma}
\begin{proof}
    We may assume that $G$ has no balanced loops. Suppose that $C$ is an unbalanced cycle and $T$ is a tree in $G$ such that $V(C) \subseteq V(T)$. Let $L'$ denote the set of loops in $\cL$. Note that $G/T$ contains an unbalanced loop $e$ that is incident with the compound vertex obtained by identifying the vertices of $T$. Observe that this vertex is a cutvertex of $G/T$. As well, by Lemma \ref{lem: vertex induced flats} either $E(G[V(T)])$ or $E(G[V(T)])\cup L'$ is a flat of $M$. 

    If $C$ belonged to $\cL$, then after contracting the loop $e$ in $G/T$, every cycle of the biased graph $G/(T\cup e)$ is balanced, and the resulting matroid is graphic. Since it contains a cutvertex, the cycle matroid of this graph is not connected by \cite[Proposition 4.1.7]{Oxley2011}. Since $E(G[V(T)]\cup L'$ is a flat of $M$ by Lemma \ref{lem: vertex induced flats}, we conclude that $M$ is not unbreakable.
    
    If $C$ belonged to $\cF$, then $M/C$ is a frame matroid. In this case, contracting the unbalanced loop $e$ of $G/T$ yields a biased graph $G'$ associated with the frame matroid $M/E(G[V(T)])$. This biased graph has more than one component having an edge. We deduce that $M/E(G[V(T)])$ is disconnected. Observe that if $L'\neq\varnothing$, then $\cF$ is degenerate. In this case, we may assume $C\in\cL$ and the result follows by the previous case. Thus we may assume $L'=\varnothing$, which implies that $E(G[V(T)])$ is a flat by Lemma \ref{lem: vertex induced flats}. We conclude that $M$ is not unbreakable.
\end{proof}

\section{Beginning the proof of the main theorem} \label{sec: when G is not 3-connected}

In this section we will prove Theorem \ref{thm: main} in the case that the underlying graph $G$ is not $3$-connected.
Throughout this section we will assume that $M = M(G, \cB,\cL,\cF)$ is a $3$-connected unbreakable quasi-graphic matroid so that each component of $G$ has at least two vertices. 

We will further assume that $M$ is not graphic, which implies that $|E(G)| \ge 4$ and not all cycles of $G$ are balanced.

The next four lemmas all hold by the $3$-connectivity of $M$.
We will use the facts that $3$-connected matroids with at least four elements contain no loops, coloops, parallel pairs, or series pairs.

\begin{lemma}
    $G$ is connected.
\end{lemma}
\begin{proof}
    Let $E = E(G)$.
    First, we note a helpful result from \cite[Corollary 4.7]{Bowler-Funk-Slilaty20}, which states that for a connected quasi-graphic matroid $M=M(G,\cB,\cL,\cF)$, either $G$ is connected or $M$ is lifted-graphic. Suppose $G$ is disconnected. By \cite[Corollary 4.7]{Bowler-Funk-Slilaty20}, we may assume $M$ is lifted graphic. 
    Recall that every component of $G$ has at least two vertices, by assumption. 
    
    If any component of $G$ is balanced, then the edges of this component do not belong to any circuits with edges of other components, contradicting that $M$ is connected. Thus we may assume that every component of $G$ contains an unbalanced cycle. 
    Let $A$ be a component of $G$.
    We see that,
    \begin{align*}
        r(A)+r(E-E(A))&\leq\lvert V(A)\rvert+\lvert V(E-E(A))\rvert-(c(G)-1)+l(E-E(A))\\
        &\leq\lvert V(G)\rvert-(c(G)-1)+l(E-E(A))\\
        &\leq r(M)+1,
    \end{align*}
    which contradicts that $M$ is 3-connected. 
\end{proof}

Since $G$ is connected and not all cycles of $G$ are balanced we see from Proposition \ref{prop: rank of X} that $r(M) = |V(G)|$.

\begin{lemma}[{\cite[Lemma 3.1]{Fife-Mayhew-Oxley-Semple2020}}]\label{lem: small cycles are unbalanced}
All $1$- and $2$-cycles in $G$ are unbalanced.
\end{lemma}
\begin{proof}
If $e$ is a balanced loop then $M$ is disconnected. Thus any $1$-cycle in $G$ must be unbalanced. Next let $X$ be a balanced $2$-cycle in $G$. Let $Y=E(G)-X$. Then,
\begin{align*}
    r(X)+r(Y)-r(X \cup Y) &\leq \lvert V(X)\rvert-1+\lvert V(Y)\rvert-\lvert V(G)\rvert \leq 1.
\end{align*}
This is a contradiction to $M$ being $3$-connected. Thus, any $2$-cycle in $G$ is unbalanced.
\end{proof}

The following result holds verbatim from \cite[Lemma 3.2]{Fife-Mayhew-Oxley-Semple2020}. Its proof is included for the convenience of verifying the ranks involved.
\begin{lemma}[{\cite[Lemma 3.2]{Fife-Mayhew-Oxley-Semple2020}}]\label{lem: min degree is at least 3}
$G$ has no vertex that meets fewer than three edges.
\end{lemma}
\begin{proof}
Let $u\in V(G)$. Suppose $\lvert E_u\rvert\leq2$. Then $r(E(G)-E_u)\leq\lvert V(G-u)\rvert=\lvert V(G)\rvert-1$, so $E_u$ contains a cocircuit of $M$. This contradicts the fact that $M$ is 3-connected having at least four elements.
\end{proof}

\begin{lemma}[{\cite[Lemma 3.3]{Fife-Mayhew-Oxley-Semple2020}}] \label{lem: G is 2-connected}
$G$ is $2$-connected.
\end{lemma}
\begin{proof}
Suppose that $G$ has a cutvertex $v$. Let $A_1$ be a component of $G-v$. Let $A$ be the graph induced by the vertex set $V(A_1)\cup v$, and let $B$ be the graph induced by the edge set $E(G)-E(A)$. Then we know that $r(E(A)) \leq |V(A)|$ and $r(E(B)) \leq |V(B)|$. Recall that, because $G$ is connected and $M$ is not graphic, we know that $r(M) = \lvert V(G)\rvert$. 
Furthermore we can see that $|V(G)| = |V(A)|+|V(B)|-1$. Thus,
\begin{align*}
    r(E(A))+r(E(B))-r(M) &\leq |V(A)|+|V(B)|-|V(G)| = 1.
\end{align*}

By Lemma \ref{lem: min degree is at least 3} we see that each of $E(A),E(B)$ has size at least two, thus the rank calculation above contradicts that $M$ is 3-connected.
\end{proof}

The following lemma differs significantly from the frame matroid case \cite[Lemma 3.4]{Fife-Mayhew-Oxley-Semple2020}.

\begin{lemma} \label{lem: si(G) is a cycle}
If $\si(G)$ is a cycle, then either $\lvert V(G)\rvert\leq6$ or $\cF$ is degenerate. 
\end{lemma}
\begin{proof}
Suppose that $\si(G)$ is a cycle and that $\lvert V(G)\rvert\geq7$. By Lemma \ref{lem: min degree is at least 3} and Lemma \ref{lem: small cycles are unbalanced}, each vertex of $G$ meets an unbalanced cycle of length at most 2. If any loops of $G$ belongs to $\cF$ then $\cL$ is degenerate, and $M=F(G,\cB)$. By \cite[Lemma 3.4]{Fife-Mayhew-Oxley-Semple2020}, $|V(G)|\leq 6$. Thus we may assume that any loops of $G$ belong to $\cL$. When $\cL$ contains a loop, $\cF$ is degenerate and the conclusion holds. Now suppose $G$ has no loops. For a vertex $v$ choose an incident unbalanced 2-cycle $C_v$. Note that, since $\lvert V(G)\rvert\geq7$, for each vertex $v$ there is a vertex $u$ so that $u$ has distance three in $G$ from each vertex of $C_v$. Moreover, for any choice of $v$ and $C_v$ there is a vertex $u$ so that each vertex of an incident unbalanced 2-cycle $C_u$ has distance two from each vertex of $C_v$. Note that because these two cycles are disjoint, either both cycles belong to $\cF$ or both belong to $\cL$. If they both belong to $\cF$, then the matroid $M/\cl(C_u\cup C_v)$ is disconnected. To see this, note that in contracting one edge from $C_u$ and one edge from $C_v$, we obtain a quasi-biased graph with two unbalanced frame-type loops at different vertices. Contracting these loops gives a disconnected graph, and it follows from Proposition \ref{prop: rank of X} that $M/\cl(C_u\cup C_v)$ is not connected, contradicting that $M$ is unbreakable. Thus both $C_u$ and $C_v$ belong to $\cL$. Since this holds for any vertex $v$, we deduce that every unbalanced 2-cycle belongs to $\cL$. From this we can observe that the only cycles in $\cF$ are Hamiltonian, and so $\cF$ is degenerate. 
\end{proof}

We can now prove the first important case of our main result, generalizing \cite[Theorem 3.5]{Fife-Mayhew-Oxley-Semple2020}.

\begin{theorem} \label{thm: 3-connected quasi matroid with G 2 connected}
Let $M = M(G, \cB, \cL, \cF)$ be a $3$-connected quasi-graphic matroid and assume that $G$ is $2$-connected but not $3$-connected.
Then either $|V(G)| \le 6$, or $\si(G)$ is a cycle and $\cF$ is degenerate.
\end{theorem}
\begin{proof}

We will follow the proof of \cite[Theorem 3.5]{Fife-Mayhew-Oxley-Semple2020} verbatim until stated otherwise.

Assume that $|V(G)| \geq 7$. Let $\{u,v\}$ be a vertex cut in $G$. Let $A_1$ and $B_1$ be disjoint non-empty graphs, each a disjoint union of components of $G-\{u,v\}$, such that $A_1 \cup B_1 = G - \{u,v\}$. Let $(A,B)$ be a partition of $E(G)$ with $A \subseteq E(G[V(A_1) \cup \{u,v\}])$ and $B \subseteq E(G[V(B_1) \cup \{u,v\}])$. Hence, each edge joining $u$ and $v$, and each unbalanced loop incident to $u$ or $v$ can lie in $A$ or $B$. Assume initially that each such edge lies in $B$. Because $M$ is 3-connected, Proposition \ref{prop: rank of X} implies that $G[A]$ is unbalanced. By symmetry, we deduce that each of $A$ and $B$ contains an unbalanced cycle that contains no edge joining $u$ and $v$ and is not an unbalanced loop incident to $u$ or $v$. 

Next we show the following.

\begin{subtheorem}[{\cite[3.5.1]{Fife-Mayhew-Oxley-Semple2020}}]\label{sub:B has a path P_B}
    Suppose B has a path $P_{B}$ joining u and v that does not use all of the vertices of B. Let $C_{A}$ be an unbalanced cycle in A, and let $P^u$ and $P^v$ be internally disjoint paths from u and v to $C_A$ with each such path using a single vertex of $C_A$. Then
    \begin{itemize}
        \item[$(i)$] $P^u$ and $P^v$ each have at most one edge;
        \item[$(ii)$] $V(C_A)\cup \{u,v\}=V(A);$
        \item[$(iii)$] $\lvert V(A)-\{u,v\}\rvert \leq2;$ and
        \item[$(iv)$] if $\lvert V(A)-\{u,v\}\rvert =2,$ then no edge in $P^u$ and $P^v$ is in a 2-cycle.
    \end{itemize}
\end{subtheorem}
\begin{proof}
The proof of \cite[3.5.1]{Fife-Mayhew-Oxley-Semple2020} only uses the matroid $M$ to show that Lemmas 2.7 and 3.2 from \cite{Fife-Mayhew-Oxley-Semple2020} hold for the biased graph $(G, \cB)$, and all other arguments only involve $(G, \cB)$.
Since we have generalized Lemmas 2.7 and 3.2 from \cite{Fife-Mayhew-Oxley-Semple2020} to Lemmas \ref{lem: tree deletion disconnects G} and \ref{lem: min degree is at least 3}, respectively, the proof of \cite[3.5.1]{Fife-Mayhew-Oxley-Semple2020} applies verbatim to prove 4.6.1 after replacing the applications of Lemmas 2.7 and 3.2 from \cite{Fife-Mayhew-Oxley-Semple2020} with Lemmas \ref{lem: tree deletion disconnects G} and \ref{lem: min degree is at least 3}, respectively.
\end{proof}

Previously, for a 2-vertex cut $\{u,v\}$ in $G$, we defined subgraphs $A$ and
$B$ whose union is $G$. If both $A$ and $B$ contain $(u,v)$-paths that do not use
all of their vertices, then, by \ref{sub:B has a path P_B}$(iii)$, $|V(G)| \le 6$. If each of $\operatorname{si}(A)$ and
$\operatorname{si}(B)$ is a path, then $\operatorname{si}(G)$ is a cycle, so by Lemma \ref{lem: si(G) is a cycle} it follows that $\cF$ is degenerate or $|V(G)| \le 6$. Thus we may assume that exactly one of $\operatorname{si}(A)$ and $\operatorname{si}(B)$
is a path. It follows that we may also assume that $G$ has no edge joining
$u$ and $v$. Now, we choose the vertex cut $\{u,v\}$ and the subgraphs $A$ and
$B$ such that $\operatorname{si}(A)$ is not a path and $|V(A)|$ is a minimum subject to this
requirement. Then $\operatorname{si}(B)$ is a path. Let $C_A$ and $C_B$ be shortest unbalanced
cycles in $A$ and $B$, respectively, such that neither cycle uses an edge joining
$u$ to $v$ and neither cycle is an unbalanced loop incident to $u$ or $v$. Subject to
this, choose $|V(C_A) \cap \{u,v\}|$ to be a maximum. By \ref{sub:B has a path P_B}$(iii)$, $|V(B)| \le 4$.
Let $P_A^u$ and $P_A^v$ be disjoint paths from $C_A$ to $u$ and $v$, respectively, chosen
so that $|P_A^u| + |P_A^v|$ is a minimum. Let $x$ and $y$ be the vertices of $C_A$ that
meet $P_A^u$ and $P_A^v$, respectively.
\begin{subtheorem}[{\cite[3.5.3]{Fife-Mayhew-Oxley-Semple2020}}]\label{sub: V(A)}
    $V(A)=V(C_A)\cup V(P^u_A)\cup V(P^v_A)$.
\end{subtheorem}
To see this, note that if $V(A) \neq V(C_A)\cup V(P^u_A)\cup V(P^v_A)$ and $T$ is a tree whose edge set is $P^u_A \cup P^v_A$ together with all but one edge of $C_A$, then $G-V(T)$ is disconnected, a contradiction to Lemma \ref{lem: vertex adjacent to every vertex of cycle}. Thus \ref{sub: V(A)} holds.
\begin{subtheorem}[{\cite[3.5.4]{Fife-Mayhew-Oxley-Semple2020}}]\label{sub:C_A does not use both}
    $C_A$ does not use both $u$ and $v$.
\end{subtheorem}
Suppose otherwise. Then, since $G$ has no edge joining $u$ and $v$, the cycle $C_A$ has a vertex not in $\{u,v\}$. By Lemma \ref{lem: min degree is at least 3}, this vertex has degree at least three, so the cycle $C_A$ is not an induced cycle of $G[A]$, a contradiction to Lemma \ref{lem: induced shortest unbalanced cycles}. Hence \ref{sub:C_A does not use both} holds.   
\begin{subtheorem}[{\cite[3.5.5]{Fife-Mayhew-Oxley-Semple2020}}]\label{sub:If u not in V(C_A)}
    If $u \notin V(C_A)$, then $\lvert P_A^u\rvert=\lvert P_A^v\rvert=1$.
\end{subtheorem}
Suppose not, letting $u'$ be the neighbor of $u$ on the path $P^u_A$. Observe that $\{u',v\}$ cannot be a vertex cut of $G$ otherwise $\si(A-u)$ is a path and so $\si(A)$ is a path, a contradiction. Thus $u$ is adjacent to some vertex $w$ of $V(A)-\{u',v\}$. By the choice of $P^u_A$, we see that $w \in V(P^v_A)$. The union of an edge joining $u$ and $w$ with the edge set of $P^v_A$ and all but one edge of $C_A$ is a tree $T$ such that $G-V(T)$ is disconnected, a contradiction to Lemma \ref{lem: vertex adjacent to every vertex of cycle}. We conclude that \ref{sub:If u not in V(C_A)} holds.

\begin{subtheorem}[{\cite[3.5.6]{Fife-Mayhew-Oxley-Semple2020}}]\label{sub:C_A geq 3}
    $\lvert C_A\rvert\geq3$.
\end{subtheorem}
This follows by \ref{sub: V(A)}, \ref{sub:If u not in V(C_A)}, and by symmetry, otherwise $|V(G)| \leq 6$.

The choice of $C_A$ implies that $A$ has no unbalanced $2$-cycles and no unbalanced loops. Hence, by Lemma \ref{lem: min degree is at least 3}, we have the following.   

\begin{subtheorem}[{\cite[3.5.7]{Fife-Mayhew-Oxley-Semple2020}}]\label{sub: every vertex of C_A is adjacent}
    Every vertex of $C_A$ must be adjacent to a vertex outside of $V(C_A)$.
\end{subtheorem}
By \ref{sub:C_A does not use both}, we may now assume that $u \notin V(C_A)$. As the next step towards proving Theorem \ref{thm: 3-connected quasi matroid with G 2 connected}, we now show the following. 

\begin{subtheorem}[{\cite[3.5.8]{Fife-Mayhew-Oxley-Semple2020}}]\label{sub: C_A geq 4}
    For $\lvert  C_A\rvert\geq 4,$ suppose $s$ and $t$ are distinct vertices of $C_A$ that are neighbors of $u$ in $G$. Then $C_A$ has an edge joining $s$ and $t$.
\end{subtheorem}

Let $f$ and $g$ be edges joining $u$ to $s$ and $t$, respectively. Then, in the
theta graph $H$ with edge set $C_A \cup \{f,g\}$, at least one cycle meeting $u$ is unbalanced.
Let $C_A'$ be such an unbalanced cycle. As $|C_A'| \le |C_A|$, equality must hold, so
there is an $(s,t)$-path $P^{st}$ in $C_A$ of length two such that the edge set of $C_A'$ is
$\{f,g\} \cup (C_A - P^{st})$. Note that $v \in V(C_A)$, otherwise, as $C_A'$ is an unbalanced
cycle in $A$ of length $|C_A|$ that uses $u$, we have a contradiction to the choice
of $C_A$. By replacing $C_A$ by $C_A'$ in \ref{sub:C_A does not use both}, we deduce that $v \notin V(C_A')$. Thus $v$
is the internal vertex of $P^{st}$. As the 4-cycle $C_A''$ with edge set $\{f,g\} \cup P^{st}$
uses $u$ and $v$, it must be balanced. By Lemma \ref{lem: min degree is at least 3} and \ref{sub: V(A)}, every vertex
in $V(C_A) - \{v,s,t\}$ is adjacent to $u$. Then, as $u$ is also adjacent to $s$ and
$t$, Lemma \ref{lem: wheel} gives us an unbalanced 3-cycle in $A$, a contradiction. We
conclude that \ref{sub: C_A geq 4} holds.

Suppose $v \in V(C_A)$. Then, by \ref{sub: every vertex of C_A is adjacent}, every vertex of $V(C_A) - v$ is adjacent
to $u$. Thus, by \ref{sub: C_A geq 4}, $|V(C_A)| \le 3$, so $|V(G)| \le 6$, a contradiction. We
may now assume that $V(C_A)$ avoids $\{u,v\}$. By \ref{sub: V(A)} and \ref{sub:If u not in V(C_A)}, $V(C_A) =
V(A) - \{u,v\}$. 

We next show that
\begin{subtheorem} [{\cite[3.5.9]{Fife-Mayhew-Oxley-Semple2020}}]\label{sub:V(B)}
    $\lvert V(B)-\{u,v\}\rvert=1$.
\end{subtheorem}
We will now stop following the proof of \cite[Theorem 3.5]{Fife-Mayhew-Oxley-Semple2020} verbatim. Recall that $\lvert V(B)\rvert\leq4.$ Suppose that $\lvert V(B)-\{u,v\}\rvert=2.$ By \ref{sub:B has a path P_B}$(i)$ and \ref{sub:B has a path P_B}$(iv)$, with the roles of $A$ and $B$ reversed, $\lvert C_B\rvert>1.$ This implies that $C_B$ is a 2-cycle that is vertex-disjoint from $\{u,v\}$. Note that both $C_A$ and $C_B$ either belong to $\cL$ or to $\cF$. First, suppose they are both in $\cF$. Then, following \cite{Fife-Mayhew-Oxley-Semple2020}, we observe that $\cl(C_B)$ consists of all edges in parallel with an edge of $C_B$ and $\cl(C_A\cup C_B)=C_A\cup \cl(C_B)$. Now, contracting $C_A\cup\cl(C_B)$ from $\qbg$ produces a 2-vertex disconnected quasi-biased graph so that each vertex meets a frame-type loop. We deduce that $M/\cl(C_A\cup C_B)$ is disconnected, a contradiction. 

Now we assume both $C_A$ and $C_B$ belong to $\cL$. We apply Lemma \ref{lem: contract a cycle} and observe the structure of $M/\cl(C_B)$. Since $\lvert C_A\rvert\geq3$ and $C_A$ does not use $u$ or $v$, note that $\si(G/C_B)$ cannot be a cycle. Moreover, by Lemma \ref{lem: induced shortest unbalanced cycles}, we have that $C_B$ is induced, so $\cl(C_B)-C$ may consist of only unbalanced loops in $\cL$. Thus by Lemma \ref{lem: contract a cycle}, $\si(G/C_B)$ is complete. This means that, in $G/C_B$, there is an edge from each vertex of $C_A$ to the vertex resulting from identifying the vertices of $C_B$. Such edges originated from edges in $G$ from $C_A$ to vertices of $C_B$, contradicting that $\{u,v\}$ was a vertex cut. We conclude that \ref{sub:V(B)} holds.

Now, as $\lvert V(G)\rvert\geq7,$ we see that $\lvert C_A\rvert\geq4$. As well, by \ref{sub: every vertex of C_A is adjacent}, every vertex of $C_A$ is adjacent to either $u$ or $v$. Note that \ref{sub: C_A geq 4} applies when considering either $u$ or $v$. Applying \ref{sub: C_A geq 4} to both $u$ and $v$, we see that $\lvert C_A\rvert=4$ and that two consecutive vertices of $C_A$ are adjacent to $u$ while the other two are adjacent to $v$. As well, by \ref{sub: C_A geq 4}, no vertex is adjacent to both $u$ and $v$. Observe that if there are any loops $e\in\cL$, then $M=L(G,\cB)$ and, since $\si(G)$ is not a cycle, by Lemma \ref{lem: contract a cycle} $M/\cl(e)$ is complete, contradicting that $\{u,v\}$ is a vertex cut. If $C_B$ is a loop in $\cF$, then $\cL$ is degenerate and $M=F(G,\cB)$. Thus the result follows from \cite[Theorem 3.5]{Fife-Mayhew-Oxley-Semple2020}. If $C_B$ is not a loop in $\cF$, then it follows that $C_B$ is a 2-cycle. Without loss of generality, this cycle must meet $u$, and we can find a tree $T$ by taking one edge of $C_B$ along with a path in $A$ of length 3 consisting of $u$, one of its neighbors in $C_A$, and both of the neighbors of $v$ in $C_A$. Deleting $T$ disconnects $G$, a contradiction to Lemma \ref{lem: tree deletion disconnects G}. This completes the proof.  
\end{proof}

\section{Finishing the proof of the main theorem} \label{sec: when G is 3-connected}

In this section we will prove the following result. Combined with Theorem \ref{thm: 3-connected quasi matroid with G 2 connected}, it proves Theorem \ref{thm: main}.

\begin{theorem} \label{thm: 4.1}
Let $M(G,\cB,\cL,\cF)$ be a $3$-connected unbreakable quasi-graphic matroid where $G$ is $3$-connected and $|V(G)| \ge 7$.
Then $\si(G)$ can be obtained from a complete graph by deleting the edges of a path of length at most two.
\end{theorem}

The following lemma generalizes \cite[Lemma 4.2]{Fife-Mayhew-Oxley-Semple2020}, and will help us understand the structure of $(G, \cB, \cL, \cF)$ from the perspective of a pair of non-adjacent vertices.

\begin{lemma} \label{lem: x,y pair of non-adjacent vertices}
Let $M=M(G,\cB,\cL,\cF)$ be a $3$-connected unbreakable quasi-graphic matroid where $G$ is $3$-connected and not all cycles are in $\cB$.
Then the following hold for any pair $\{x,y\}$ of non-adjacent vertices of $G$:
\begin{enumerate}[$(i)$]
\item $G - \{x,y\}$ is balanced.

\item Every unbalanced cycle in $G$ uses at least one of $x$ and $y$.

\item There is at least one unbalanced cycle in $G$ that avoids $x$ and at least one unbalanced cycle that avoids $y$.

\item If $C_y$ is a shortest unbalanced cycle in $G$ containing $y$ and avoiding $x$, and $|C_y| \ge 3$, then $C_y$ is an induced subgraph of $G$.
\end{enumerate}
\end{lemma}
\begin{proof}
The proof of parts $(i)$, $(ii)$, and $(iv)$ of \cite[Lemma 4.2]{Fife-Mayhew-Oxley-Semple2020} only uses the matroid $M$ to show that Lemmas 2.4 and 2.7 from \cite{Fife-Mayhew-Oxley-Semple2020} apply to the biased graph $(G,\cB)$. Since we have generalized Lemmas 2.4 and 2.7 from \cite{Fife-Mayhew-Oxley-Semple2020} to Lemmas 3.4 and 3.7, respectively, the proofs of parts $(i),(ii)$, and $(iv)$ apply verbatim to prove Lemma \ref{lem: x,y pair of non-adjacent vertices} after replacing the applications of Lemmas 2.4 and 2.7 from \cite{Fife-Mayhew-Oxley-Semple2020} with Lemmas 3.4 and 3.7, respectively.

Part $(iii)$ will require contraction to prove. 
Suppose $G-x$ is balanced. 
Then the edge set $W$ of $G - \{x,y\}$ is a flat of $M$.
The graph $G/W$ has three vertices including a cutvertex that results from identifying all the vertices in $W$. 
In $G/W$, all the cycles incident with $y$ are balanced, so this cutvertex actually induces a separation in $M/W$, a contradiction.
Thus $(iii)$ holds.
\end{proof}

The following is an immediate consequence of the last lemma using parts $(i)$ and $(ii)$. It generalizes \cite[Lemma 4.3]{Fife-Mayhew-Oxley-Semple2020}.

\begin{lemma}\label{lem: delete a cycle}
    Let $M=M\qbg$ be a $3$-connected unbreakable quasi-graphic matroid where $G$ is $3$-connected. Suppose $G$ has a pair of non-adjacent vertices $\{x,y\}$. Then $\si(G-V(C))$ is complete for every unbalanced cycle $C$.
\end{lemma}

For the rest of this section, we will consider $u$ and $v$ to be a fixed non-adjacent pair of vertices of $G$, and $W$ will denote $E(G-\{u,v\})$. We will write $C_u$ to be the shortest unbalanced cycle avoiding $v$. By Lemma $\ref{lem: x,y pair of non-adjacent vertices}(ii)$, such a cycle must use $u$. Similarly $C_v$ will be the shortest unbalanced cycle avoiding $u$, which must use $v$. The following critical lemma generalizes \cite[Lemma 4.4]{Fife-Mayhew-Oxley-Semple2020} from frame matroids to quasi-graphic matroids. 
We will only use parts $(ii)$ and $(vi)$.

\begin{lemma}\label{lem: 6 part lemma}
    Let $M=M(G,\cB,\cL,\cF)$ be a $3$-connected unbreakable quasi-graphic matroid where $G$ is $3$-connected and has at least one unbalanced cycle. Suppose that $|C_u| \geq |C_v|$. 
    \begin{enumerate}[$(i)$]
        \item Suppose that $|C_u| \geq 4$ and $C$ is an unbalanced cycle of $G$ that avoids $u$. Then $C$ uses all but at most one vertex of $V(C_u) - u$. Moreover, if there is a vertex in $(V(C_u)-u)-V(C)$, then it must be adjacent to $u$. 
        \item $|C_u| \leq 4$. 
        \item Either $\lvert C_v\rvert\in \{1,2\}$, or the subgraph of $G$ induced by $C_u\cup C_v$ is one of the graphs shown in \cite[Figure 1]{Fife-Mayhew-Oxley-Semple2020}.
        \item If $w\in V(G)$ is in an unbalanced cycle of size at most three, then $w$ is nonadjacent to at most two other vertices.
        \item Every vertex $w$ of $G$ is nonadjacent to at most three other vertices.
        \item If $\lvert C_u\rvert=4$, then $\lvert V(G)\rvert\leq6$ and $\si(G)$ has at most seven edges fewer than the complete graph on $\lvert V(G)\rvert$ vertices.  
\end{enumerate}
\end{lemma}
\begin{proof}
The proof of \cite[Lemma 4.4]{Fife-Mayhew-Oxley-Semple2020} only uses the matroid $M$ to show that Lemmas 3.2, 4.2, and 4.3 from \cite{Fife-Mayhew-Oxley-Semple2020} hold for the biased graph $(G, \cB)$, and all other arguments only involve $(G, \cB)$.
Since we have generalized Lemmas 3.2, 4.2 and 4.3 from \cite{Fife-Mayhew-Oxley-Semple2020} to Lemmas \ref{lem: min degree is at least 3}, \ref{lem: x,y pair of non-adjacent vertices} and \ref{lem: delete a cycle}, respectively, the proof of \cite[Lemma 4.4]{Fife-Mayhew-Oxley-Semple2020} applies verbatim to prove Lemma \ref{lem: 6 part lemma} after replacing the applications of Lemmas 3.2, 4.2, and 4.3 from \cite{Fife-Mayhew-Oxley-Semple2020} with Lemmas \ref{lem: min degree is at least 3}, \ref{lem: x,y pair of non-adjacent vertices}, and \ref{lem: delete a cycle}, respectively.
\end{proof}


We can now prove Theorem \ref{thm: 4.1}.

\begin{proof}[Proof of Theorem 5.1]
    We will follow the proof of \cite[Theorem 4.1]{Fife-Mayhew-Oxley-Semple2020} verbatim until stated otherwise.
    
    \begin{subtheorem}[{\cite[4.4.1]{Fife-Mayhew-Oxley-Semple2020}}]\label{H is a simple 3 connected graph on 7 vertices claw case} 
        Let $H$ be a simple $3$-connected graph on at least seven vertices. Let $u, v_1, v_2$ and $v_3$ be distinct vertices of $H$ such that $H$ has none of the edges $(u,v_1), (u,v_2)$ and $(u,v_3)$. Then $H \neq \si(G)$. 
    \end{subtheorem}
    \begin{proof}
    The proof of \cite[4.4.1]{Fife-Mayhew-Oxley-Semple2020} uses Lemmas 4.3 and 4.4 from \cite{Fife-Mayhew-Oxley-Semple2020} and otherwise only uses properties of $(G, \cB)$.
Since we have generalized Lemmas 4.3 and 4.4 from \cite{Fife-Mayhew-Oxley-Semple2020} to Lemma \ref{lem: delete a cycle} and \ref{lem: 6 part lemma}, respectively, \ref{H is a simple 3 connected graph on 7 vertices claw case} holds.
    \end{proof}
    
    Next we show the following.
    
    \begin{subtheorem}[{\cite[4.4.2]{Fife-Mayhew-Oxley-Semple2020}}] \label{H is a simple 3 connected graph on 7 vertices matching case}
        Suppose that $H$ is a simple $3$-connected graph with at least seven vertices. If, for distinct vertices $u,v,s$, and $t$, neither $(u,v)$ nor $(s,t)$ is an edge of $H$, then $\si(G) \neq H$.  
    \end{subtheorem}

Suppose that $\si(G) = H$. Let $C_u$, $C_v$, $C_s$, and $C_t$ be shortest unbalanced cycles avoiding, respectively, $v$, $u$, $t$, and $s$. By Lemma \ref{lem: delete a cycle}, these cycles use, respectively, $u$, $v$, $s$, and $t$. Moreover, by Lemma \ref{lem: 6 part lemma}$(ii)$ and Lemma \ref{lem: 6 part lemma}$(vi)$, all of these cycles have at most three edges. Hence each of $C_u$ and $C_v$ meets exactly one vertex of $\{s, t\}$, and each of $C_s$ and $C_t$ meets exactly one vertex of $\{u, v\}$. We would like circuits $C_1$ and $C_2$ in $\{C_u, C_v, C_s, C_t\}$ with $\{u, v, s, t\}\subseteq C_1\cup C_2$. Now $C_u$ uses $u$ and avoids $v$, while $C_v$ uses $v$ and avoids $u$. Both $C_u$ and $C_v$ use exactly one vertex of $\{s, t\}$. We can take $C_1= C_u$ and $C_2 = C_v$ unless, by symmetry, $C_u$ and $C_v$ both contain $s$. Now $C_t$ uses $u$ or $v$ but not both. Taking $C_2 = C_t$, we let $C_1$ be $C_v$ or $C_u$, respectively. By potentially relabelling $s$ and $t$, we may assume that $V(C_1)$ and $V(C_2)$ meet $\{u, v, s, t\}$ in $\{u, s\}$ and $\{v, t\}$, respectively.

Continuing with the proof of \ref{H is a simple 3 connected graph on 7 vertices matching case}, we now show the following.
\begin{subtheorem}[{\cite[4.4.3]{Fife-Mayhew-Oxley-Semple2020}}] \label{vertex adjacent to every vertex of G-y} 
    There is a vertex $y\in V(G)-(V(C_1)\cup V(C_2))$ that is adjacent to each vertex of $G-y$.
\end{subtheorem}
\begin{proof}
The proof of \cite[4.4.3]{Fife-Mayhew-Oxley-Semple2020} uses Lemmas 4.3 and 4.4 from \cite{Fife-Mayhew-Oxley-Semple2020} and otherwise only uses properties of $(G, \cB)$.
Since we have generalized Lemmas 4.3 and 4.4 from \cite{Fife-Mayhew-Oxley-Semple2020} to Lemma \ref{lem: delete a cycle} and \ref{lem: 6 part lemma}, respectively, \ref{vertex adjacent to every vertex of G-y} holds.
\end{proof}

For a vertex $z$, recall that $E_z$ is the set of edges meeting $z$. Let $X$ be the set of edges that only meet vertices in $\{u, v, s, t\}$. 

We show next that
\begin{subtheorem}[{\cite[4.4.4]{Fife-Mayhew-Oxley-Semple2020}}] \label{subthm: G delete X is balanced}
    $G\del X$ has all of its cycles balanced.
\end{subtheorem}
Suppose that $G\del X$ has an unbalanced cycle $C$. Since $y$ is adjacent to each vertex of $C-y$, by Lemma \ref{lem: vertex adjacent to every vertex of cycle}, $G$ has an unbalanced 3-cycle $C_y$ with $C_y\subseteq C\cup E_y$. Let $f$ be the edge of $C_y$ that is not incident with $y$. Then $f\in C$ and, by assumption, $f$ does not join two vertices of $\{u,v,s,t\}$. Thus $C_y$ avoids at least three vertices in $\{u,v,s,t\}$. But, by Lemma \ref{lem: x,y pair of non-adjacent vertices}$(ii)$, $C_y$ meets both $\{u,v\}$ and $\{s,t\}$, a contradiction. We deduce that \ref{subthm: G delete X is balanced} holds. 

Since $y$ is adjacent to each vertex of the unbalanced cycle $C_1$, by Lemma \ref{lem: vertex adjacent to every vertex of cycle}, there is an unbalanced 3-cycle $C'$ using $y$ and exactly two vertices of $C_1$. Because neither $(u,v)$ nor $(s,t)$ is an edge of $G$, Lemma \ref{lem: x,y pair of non-adjacent vertices}$(ii)$ implies that $V(C')$ contains $u$ and $s$. Hence $V(C')=\{y,u,s\}$. By symmetry, there is an unbalanced cycle $C''$ with vertex set $\{y,v,t\}$. 

We now stop following the proof of \cite[Theorem 4.1]{Fife-Mayhew-Oxley-Semple2020} verbatim.
Let $F$ be the flat of $M$ that is spanned by the edges meeting $y$ and one of $u,v,s$, and $t$. 
The edges meeting $y$ contain no unbalanced cycles by \ref{subthm: G delete X is balanced}, so $F$ contains no unbalanced cycles. 
The biased graph $G'$ corresponding to $M/F$ has unbalanced loops at $y$ corresponding to the edges $(u,s)$ and $(v,t)$ of $G$. 

As $G\del X$ has every cycle balanced, letting $X'=X-F$, we deduce that $G'\del X'$ has only balanced cycles. Thus $M/F$ has no circuit that meets both $X'$ and $E(G'\del X')$. As the last two sets are non-empty, this implies that $M/F$ is disconnected, which contradicts the fact that $M$ is unbreakable. We conclude that \ref{H is a simple 3 connected graph on 7 vertices matching case} holds.

By \ref{H is a simple 3 connected graph on 7 vertices claw case}, the complement of $\si(G)$ in $K_n$ has no vertex of degree three or more and, by \ref{H is a simple 3 connected graph on 7 vertices matching case}, has no two-edge matching. Thus this complement is a path of length at most two. Hence Theorem \ref{thm: 4.1} holds.
\end{proof}

We now show that the lower bound $|V(G)| > 6$ is sharp in Theorems \ref{thm: main} and \ref{thm: main corollary for lifted-graphic matroids}.
Let $G$ be a loopless, $2$-connected, $3$-regular, $6$-vertex graph with nine edges such that each theta subgraph has at least five edges, and for every $2$-edge cut $\{x,y\}$, both components of $G \del \{x,y\}$ have a cycle.
Some examples of graphs of this form are shown in Figure \ref{figure: 6-vertex graphs}.
Let $M$ be the frame matroid of $G$ with all cycles unbalanced.
We claim that $M$ is $3$-connected and unbreakable.
Using Proposition \ref{prop: rank of X} we see that $M$ has no coloops or series pairs, so $M^*$ is a simple rank-$3$ matroid.
Since $G$ is loopless, $3$-regular, and each theta subgraph has at least five elements, each circuit of $M$ has at least five elements. 
Therefore each hyperplane of $M^*$ has at most four elements, so $M^*$ is not the union of two hyperplanes.
It follows easily that $M^*$ is $3$-connected, and therefore $M$ is $3$-connected.
Since a simple rank-$3$ matroid cannot have skew circuits, it follows from \cite[Theorem 1.1]{Oxley-Pfeil2022} that $M$ is unbreakable.
Since none of the graphs in Figure \ref{figure: 6-vertex graphs} simplify to a cycle or a complete graph missing a path with at most two edges, the bound $|V(G)| > 6$ of Theorem \ref{thm: main} is sharp.
Since the lift matroid and frame matroid of the first graph in Figure \ref{figure: 6-vertex graphs} with all cycles unbalanced are equal, the bound $|V(G)| > 6$ of Corollary \ref{thm: main corollary for lifted-graphic matroids} is sharp as well.

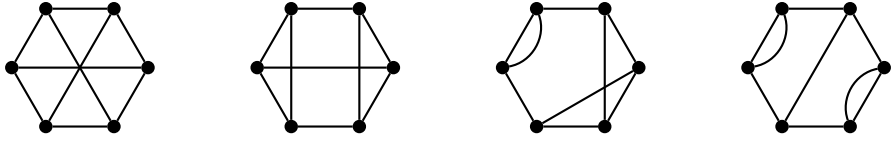
\begin{figure} 
\centering
\begin{tabular}{ccc ccc ccc c}

\begin{tikzpicture}[thick, every node/.style={inner sep=1.5pt,circle,draw,fill}]
     \node[label = {}] (1) at (0:\R) {};
     \node (2) at (60:\R) {};
     \node (3) at (120:\R) {};
     \node (4) at (180:\R) {};
     \node (5) at (240:\R) {};
      \node (6) at (300:\R) {};
     \draw (1) -- (2);
\draw (2) -- (3);
\draw (3) -- (4);
\draw (4) -- (5);
\draw (5) -- (6);
\draw (6) -- (1);
\draw (2) -- (5);
\draw (3) -- (6);
\draw (4) -- (1);
\end{tikzpicture}

&&& \begin{tikzpicture}[thick, every node/.style={inner sep=1.5pt,circle,draw,fill}]
     \node[label = {}] (1) at (0:\R) {};
     \node (2) at (60:\R) {};
     \node (3) at (120:\R) {};
     \node (4) at (180:\R) {};
     \node (5) at (240:\R) {};
      \node (6) at (300:\R) {};
     \draw (1) -- (2);
\draw (2) -- (3);
\draw (3) -- (4);
\draw (4) -- (5);
\draw (5) -- (6);
\draw (6) -- (1);
\draw (3) -- (5);
\draw (1) -- (4);
\draw (2) -- (6);
\end{tikzpicture}

&&& \begin{tikzpicture}[thick, every node/.style={inner sep=1.5pt,circle,draw,fill}]
     \node[label = {}] (1) at (0:\R) {};
     \node (2) at (60:\R) {};
     \node (3) at (120:\R) {};
     \node (4) at (180:\R) {};
     \node (5) at (240:\R) {};
      \node (6) at (300:\R) {};
     \draw (1) -- (2);
\draw (2) -- (3);
\draw (3) -- (4);
\draw (4) -- (5);
\draw (5) -- (6);
\draw (6) -- (1);
\draw (2) -- (6);
\draw (5) -- (1);
\draw (3) to[bend left=50] (4);
\end{tikzpicture}

&&& \begin{tikzpicture}[thick, every node/.style={inner sep=1.5pt,circle,draw,fill}]
     \node[label = {}] (1) at (0:\R) {};
     \node (2) at (60:\R) {};
     \node (3) at (120:\R) {};
     \node (4) at (180:\R) {};
     \node (5) at (240:\R) {};
      \node (6) at (300:\R) {};
     \draw (1) -- (2);
\draw (2) -- (3);
\draw (3) -- (4);
\draw (4) -- (5);
\draw (5) -- (6);
\draw (6) -- (1);
\draw (5) -- (2);
\draw (3) to[bend left=50] (4);
\draw (6) to[bend left=50] (1);
\end{tikzpicture}

\end{tabular}
\caption{Graphs showing sharpness of the bound $|V(G)| > 6$ in Theorem \ref{thm: main}.}
\label{figure: 6-vertex graphs}
\end{figure}

\section{The complete characterization} \label{sec: characterization}

We have shown that if $M = M(G, \cB, \cL, \cF)$ is $3$-connected and unbreakable, then $\si(G)$ is a nearly complete graph or a cycle (assuming each component of $G$ has at least two vertices and $|V(G)| > 6$), proving Theorem \ref{thm: main}. 
In this section we will refine this result by giving a complete characterization of the $3$-connected unbreakable quasi-graphic matroids.
We will prove two preliminary lemmas, and then we will separately consider two cases depending on whether $\si(G)$ is a nearly complete graph or a cycle.
We begin with a straightforward observation about unbreakable matroids.

\begin{lemma} \label{lem: contract to rank 2}
    If $M$ is a connected matroid that is not unbreakable, then $M$ has a flat $F$ so that $M/F$ is disconnected and $r(M/F) = 2$.
\end{lemma}
\begin{proof}
If $M$ is connected but not unbreakable then there exists a flat $F'$ of $M$ such that $M/F'$ is not connected. Say $M/F'$ decomposes into a direct sum of connected matroids as follows: $M/F'= M_1 \oplus M_2 \oplus \ldots \oplus M_k$ for some $k \geq 2$. Suppose that $r(M/F') \neq 2$. Observe that because $M/F'$ is loopless, $r(M_i)\geq1$ for all $i=1,2,\ldots,k$, which implies that $r(M/F')> 2$. Set $F''$ to be a flat of $M/F'$ consisting of the elements of $\cup_{i \geq 3} E(M_i)$ and a hyperplane $H_1$ in $M_1$ and $H_2$ in $M_2$, respectively. Since $F''$ is a flat of $E(M)-F'$, the set $F' \cup F''$ is a flat of $M$ by \cite[Proposition~3.3.7]{Oxley2011}. We can contract the flat $F' \cup F''$ to obtain a decomposition as follows: $M/(F' \cup F'') = M_1\big/H_1\oplus M_2\big/H_2$ where $r(M_1/H_1)=1$ and $r(M_2/H_2)=1$. This proves the statement with $F = F' \cup F''$.
\end{proof}

\begin{figure} 
\centering
\setlength{\tabcolsep}{2pt} 
\begin{tabular}{ccc ccc ccc c}

\begin{tikzpicture}[thick, every node/.style={inner sep=1.5pt,circle,draw,fill}]

\node (a) at (0,0) {};
\node (b) at (2,0) {};
\node (c) at (1,-0.8) {};
\foreach \y in {0.1,0.5,0.9} {
    \draw (a) .. controls (0.5,-\y) .. (c);
  }
\foreach \y in {0.1,0.5,0.9} {
    \draw (b) .. controls (1.5,-\y) .. (c);
  }  
\node[draw=none, fill=none] at (1,-1.5 ) {\small $(i)$};
\end{tikzpicture}

&&&\begin{tikzpicture}[thick, every node/.style={inner sep=1.5pt,circle,draw,fill}]
  \node (a) at (0,0) {};
  \node (b) at (2,0) {};
  \foreach \y in {-0.8, -0.6, -0.4} {
    \draw (a) .. controls (1,-\y) .. (b);
  }
  \foreach \y in {-0.8, -0.6, -0.4} {
    \draw[densely dotted] (a) .. controls (1,\y) .. (b);
  }
  \node[draw=none, fill=none] at (1,-1.1) {\small $(ii)$};
\end{tikzpicture}

&&&\begin{tikzpicture}[thick, every node/.style={inner sep=1.5pt,circle,draw,fill}]
  \node (a) at (0,0) {};
  \node (b) at (2,0) {};
    \draw[thick]   (a) .. controls +(120:1cm) and +(-150:1cm) .. (a);
\draw[thick]   (a) .. controls +(120:1cm) and +(30:1cm) .. (a);
\draw[thick]   (a) .. controls +(210:1cm) and +(-50:1cm) .. (a);

  \foreach \y in {-0.5, -0.25, 0, 0.25, 0.5} {
    \draw (a) .. controls (1,\y) .. (b);
  }
  \node[draw=none, fill=none] at (1,-1.0) {\small $(iii)$};
\end{tikzpicture}

&&&\begin{tikzpicture}[thick, every node/.style={inner sep=1.5pt,circle,draw,fill}]
  \node (a) at (0,0) {};
  \node (b) at (2,0) {};
  \foreach \y in {-0.5, -0.25, 0, 0.25, 0.5} {
    \draw (a) .. controls (1,\y) .. (b);
  }
  \draw[thick]   (a) .. controls +(120:1cm) and +(-150:1cm) .. (a);
\draw[thick]   (a) .. controls +(120:1cm) and +(30:1cm) .. (a);
\draw[thick]   (a) .. controls +(210:1cm) and +(-50:1cm) .. (a);
\draw[thick]   (b) .. controls +(-40:1cm) and +(40:1cm) .. (b);
\draw[thick]   (b) .. controls +(120:1cm) and +(30:1cm) .. (b);
\draw[thick]   (b) .. controls +(210:1cm) and +(-50:1cm) .. (b);
   \node[draw=none, fill=none] at (1,-1.0) {\small $(iv)$};
\end{tikzpicture}

\end{tabular}
\caption{Graphs $(i)$--$(iv)$ illustrate cases $(i)$--$(iv)$ of Lemma \ref{lem: rank-2 disconnected loopless quasi-biased graphs}.}
\label{fig: small graphs}
\end{figure}
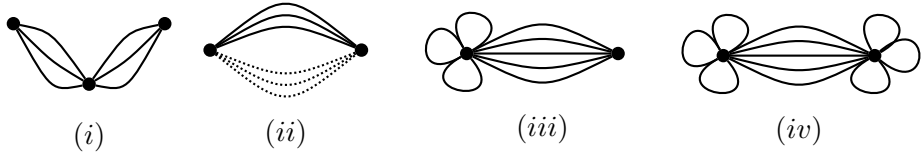

We will next characterize all quasi-biased graphs $(G, \cB, \cL, \cF)$ for which $G$ is connected and $M(G, \cB, \cL, \cF)$ is disconnected, loopless, and has rank two.
Note that every disconnected, loopless, rank-$2$ matroid simplifies to $U_{2,2}$ and is therefore the union of two parallel classes.
See Figure \ref{fig: small graphs} for illustrations of quasi-biased graphs satisfying the four outcomes of the following lemma.

\begin{lemma} \label{lem: rank-2 disconnected loopless quasi-biased graphs}
    Let $(G, \cB, \cL, \cF)$ be a quasi-biased graph so that $G$ is connected and $M(G, \cB, \cL, \cF)$ is disconnected, loopless, and has rank two. 
    Let $X$ and $Y$ be the two parallel classes of $M(G, \cB, \cL, \cF)$.
    Then one of the following holds:

    \begin{enumerate}[$(i)$]
        \item $|V(G)| = 3$, $G$ is loopless, all cycles are balanced, and $X$ and $Y$ are parallel classes of edges between different pairs of vertices.

        \item $|V(G)| = 2$, $G$ is loopless, and a cycle is balanced if and only if it is contained in $X$ or contained in $Y$.

        \item $|V(G)| = 2$, $X$ consists of unbalanced loops at a common vertex, and $Y$ consists of all edges between the two vertices of $G$.
    
        \item $|V(G)| = 2$, $G$ has loops at both vertices, all loops are in $\cL$, $X$ is the set of loops, and $Y$ is the set of non-loops.
    \end{enumerate}
\end{lemma}
\begin{proof}
    If edges $e$ and $f$ of $G$ are parallel in $M(G, \cB, \cL, \cF)$, they are either both loops of $G$ or both non-loops of $G$. We will consider several cases depending on whether the two parallel classes consist of loops or non-loops. 
    
    If both parallel classes only consist of non-loops, then $|V(G)|$ is $2$ or $3$ since $G$ is connected. If $|V(G)|=3$ then $(i)$ holds and if $|V(G)|=2$ then $(ii)$ holds.

    Note that both parallel classes cannot consist of loops. Suppose to the contrary that parallel classes $X$ and $Y$ are both collections of loops. Then these loops must be distributed across at least two vertices $x$ and $y$. To see this, observe that since none of these loops are balanced, a collection of loops at a single vertex forms only one parallel class. Thus the loops of $X$ are at a vertex $x$ and the loops of $Y$ are at a vertex $y$, which contradicts that $G$ is connected.

    Now suppose $X$ consists of loops and $Y$ consists of non-loops. Note that each loop is unbalanced. If the loops in $X$ are all at a common vertex, then $Y$ consists of all edges between a pair of vertices. Since $G$ is connected and $E(G)=X\cup Y$, the edges of $Y$ are incident with the vertex of $X$, and $(iii)$ holds. Now suppose the edges of $X$ are distributed across two vertices $\{u,v\}$. All loops must be in $\cL$ since $X$ is a parallel class. Moreover, the edges of $Y$ must be incident to both $u$ and $v$ since $G$ is connected, so $(iv)$ holds.
\end{proof}

\subsection{Nearly complete graphs}\label{sec: nearly complete graphs}

We will now show that if $M = M(G, \cB, \cL, \cF)$ is connected but not unbreakable and $\si(G)$ is obtained from a complete graph by deleting a path with at most two edges, then $(G, \cB, \cL, \cF)$ has a highly structured balancing set. 
By Lemma \ref{lem: contract to rank 2}, $M$ has a flat $F$ so that $M/F$ has rank two and is disconnected.
We will take this minor and consider the structure that it imposes on $(G, \cB, \cL, \cF)$.
For disjoint sets $V_1$ and $V_2$ of vertices we will write $E[V_1, V_2]$ for the set of edges with one end in $V_1$ and one end in $V_2$.

\begin{theorem} \label{thm: structure of breakable complete quasi-graphic matroids}
Let $M = M(G, \cB, \cL, \cF)$ be a connected quasi-graphic matroid so that $\si(G)$ is obtained from a complete graph by deleting a path with at most two edges, and $|V(G)| > 3$.
Then $M$ is not unbreakable if and only if one of the following holds:
 \begin{itemize}
        \item[$(i)$] $V(G)$ has a partition $(V_1, V_2, V_3)$ so that $G[V_i]$ is connected for $i = 1,2,3$, there are no edges between $V_1$ and $V_2$, and either $G[V_i]$ contains a cycle in $\cL$ for some $i \in \{1,2,3\}$, or every cycle of $G$ is balanced.

        \item[$(ii)$] $V(G)$ has a partition $(V_1, V_2)$ so that $G[V_1]$ and $G[V_2]$ are connected and $E[V_1, V_2]$ has a partition into two balancing sets.

        \item[$(iii)$] $V(G)$ has a partition $(V_1, V_2)$ and $(G, \cB, \cL, \cF)$ has a balancing set $X$ so that $X \subseteq E(G[V_1])$ and $G[V_1] \del X$ and $G[V_2]$ are both connected.

        \item[$(iii)'$] $V(G)$ has a partition $(V_1, V_2, V_3)$ and $(G, \cB, \cL, \cF)$ has a balancing set $X$ so that $X \subseteq E(G[V_1 \cup V_3])$, the graphs $G[V_1]\del X$, $G[V_2]$, and $G[V_3] \del X$ are connected, $G[V_3]$ has a cycle in $\cF$, and there are no edges between $V_2$ and $V_3$.

        \item[$(iv)$] $V(G)$ has a partition $(V_1, V_2)$ and $(G, \cB, \cL, \cF)$ has a balancing set $X$ so that $X \subseteq E(G[V_1]) \cup E(G[V_2])$ and $G[V_i]\del X$ is connected for $i = 1,2$, and $G[V_i]$ has a cycle in $\cL$ for $i = 1,2$.

        \item[$(v)$] $V(G)$ has a partition $(V_1, V_2, V_3)$ so that $G[V_i]$ is connected for $i = 1,2,3$, there are no edges between $V_1$ and $V_2$, and $G[V_3]$ contains a cycle in $\cF$.
    \end{itemize}
\end{theorem}

\begin{proof}
We will first show that if one of the above outcomes holds, then $M$ is not unbreakable.
For each outcome other than $(v)$ we will find a flat $F$ of $M$ so that $M/F$ gives one of the disconnected minors from Lemma \ref{lem: rank-2 disconnected loopless quasi-biased graphs}. 

First suppose $(G,\cB, \cL,\cF)$ is as in case $(i)$. Let $F=\cl(E(G[V_1])\cup E(G[V_2])\cup E(G[V_3]))$. If every cycle of $G$ is balanced, then $M$ is graphic and $\qbg/F$ has the form of Lemma \ref{lem: rank-2 disconnected loopless quasi-biased graphs}$(i)$. If $G[V_i]$ contains a cycle $C$ in $\cL$ for some $i \in \{1,2,3\}$, then first contracting this cycle gives a quasi-biased graph with no unbalanced cycles. Contracting the remaining edges of $F-C$ results in a quasi-biased graph $\qbg/F$ of the form of Lemma \ref{lem: rank-2 disconnected loopless quasi-biased graphs}$(i)$.

Next suppose $(G,\cB, \cL,\cF)$ is as in case $(ii)$. Let $F=\cl(E(G[V_1]) \cup E(G[V_2]))$. Since every cycle of $G[V_1]$ or $G[V_2]$ is balanced and each of $G[V_1]$ and $G[V_2]$ is connected, and $E[V_1,V_2]$ has a partition into two balancing sets $X$ and $Y$, then $G\del X$ and $G\del Y$ are balanced. Thus $G/F$ contains these two balancing sets $X$ and $Y$ as two distinct parallel classes. These parallel classes consist of parallel edges between the vertices represented by the contraction of $G[V_1]$ and $G[V_2]$ respectively, so that each 2-cycle of $X$ or $Y$ is balanced, as in the form of Lemma \ref{lem: rank-2 disconnected loopless quasi-biased graphs}$(ii)$.

Next suppose $(G,\cB, \cL,\cF)$ is as in case $(iii)$. Let $F=\cl((E(G[V_1])- X) \cup E(G[V_2]))$. After contraction of $F$, the edges in $X$ become unbalanced loops on the contracted vertex obtained from $G[V_1]\del X$. These loops create one parallel class of $M/F$. The contracted vertex obtained from $G[V_2]$ does not have any loops on it. The edges in $E[V_1,V_2]$ form a parallel class of $M/F$ that consists of balanced 2-cycles. This results in the quasi-biased graph of Lemma \ref{lem: rank-2 disconnected loopless quasi-biased graphs}$(iii)$.

Next suppose $(G,\cB, \cL,\cF)$ is as in $(iii)'$. Let $F=\cl(E(G[V_1]\del X) \cup E(G[V_2])\cup E(G[V_3]))$. Since $G[V_i]\del X$ is connected for $i=1,3$, we may first contract spanning trees $T_1, T_2, T_3$ of $G[V_1]\del X,G[V_2]$, and $G[V_3]\del X$ respectively so that $\qbg/(T_1\cup T_2\cup T_3)$ has three vertices $v_1, v_2, v_3$ obtained by identifying the vertices of $V_1,V_2,$ and $V_3$ respectively. Since $X$ is a balancing set, $E[v_1,v_2]$ is balanced. Moreover, by Lemma \ref{lem: at least one frame-type loop}, one loop at $v_3$ is a frame-type loop since $G[V_3]$ contains a cycle in $\cF$. Since there are no edges between $V_2$ and $V_3$, there are no edges between $v_2$ and $v_3$. Contracting the remaining unbalanced loops at $v_3$, each edge of $E[v_1,v_3]$ is deleted and replaced with an unbalanced frame-type loop at $v_1$, so that we have one parallel class consisting of $E[v_1,v_2]$ and one consisting of the unbalanced loops at $v_1$, which gives us the quasi-biased graph of Lemma \ref{lem: rank-2 disconnected loopless quasi-biased graphs}$(iii)$ after deleting the isolated vertex $v_3$.

For case $(iv)$, let $F=\cl((E(G[V_1]) \cup E(G[V_2]))- X)$. All loops of $\qbg/F$ are unbalanced and, since $G[V_i]$ has a cycle in $\cL$ for $i=1,2$, one of them is a lift-type loop. These loops, which are the edges of $X \cap G[V_i]$ for $i=1,2$, appear on the two vertices obtained from contracting $G[V_i]\del X$ for $i=1,2$. Since the unbalanced loops at different vertices are disjoint, they share the same unbalanced cycle type in $\qbg/F$. Thus all loops of $\qbg/F$ are lift-type and create one parallel class of $M/F$. The edges of $E[V_1,V_2]$ create the other parallel class of $M/F$, since $X-F$ is a balancing set of $(G,\cB,\cL,\cF)/F$. This results in the quasi-biased graph of Lemma \ref{lem: rank-2 disconnected loopless quasi-biased graphs}$(iv)$. 

Finally, in the case of outcome $(v)$, we choose $F=\cl(E(G[V_1])\cup E(G[V_2])\cup E(G[V_3]))$. Observe that by first contracting a spanning tree $T_i$ of $G[V_i]$ for $i = 1,2,3$, this leaves each of the parts $V_i$ identified as a vertex $v_i$ in $\qbg/F$ for $i=1,2,3$, with some possible loops on each, so that there is no edge between $v_1,v_2$. Since $G[V_3]$ contains a cycle in $\cF$, after contracting the edges of $T_1\cup T_2\cup T_3$ we see that $v_3$ is incident with a frame-type loop. After contracting $F-(T_1\cup T_2\cup T_3)$, each edge of $E[V_1,V_3]$ or $E[V_2,V_3]$ is deleted and replaced with a frame-type loop on $v_1$ or $v_2$, respectively, in $\qbg/F$. Since the graph of $\qbg/F$ is disconnected with frame-type cycles in different components, it follows from Proposition \ref{prop: rank of X} that $M/F$ is disconnected. 

Now we will show that if $M$ is connected but not unbreakable, then one of the theorem statement outcomes holds. By Lemma \ref{lem: contract to rank 2}, $M$ has a flat $F$ so that $M/F$ has rank $2$ and is disconnected. 
Let $(G', \cB', \cL', \cF')$ be obtained from $(G, \cB, \cL, \cF)/F$ by deleting isolated vertices. We consider two cases depending on the connectivity of $G'$, beginning with the cases in which $G'$ is connected, which all arise from Lemma \ref{lem: rank-2 disconnected loopless quasi-biased graphs}.

First suppose that $(G',\cB',\cL',\cF')$ is the quasi-biased graph of Lemma \ref{lem: rank-2 disconnected loopless quasi-biased graphs}$(i)$. We will show that $\qbg$ satisfies outcome $(i)$. Let $v_1,v_2,$ and $v_3$ be the vertices of $G'$ so that there are no edges between $v_1$ and $v_2$. By Lemma \ref{lem: at least one frame-type loop}, $F$ contains no cycles in $\cF$, so $G'=G/F$. Then the vertices $v_1,v_2,v_3$ of $G'$ correspond to a partition $V_1,V_2,V_3$ of $V(G)$ so that $G[V_i]$ is connected for $i=1,2,3$ and there are no edges between $V_1$ and $V_2$. 
Since all cycles of $\qbgp$ are balanced, either every cycle of $G$ is balanced or one of $G[V_i]$ contains a cycle in $\cL$. Otherwise, there is an unbalanced cycle $C$ using edges between $E[V_1,V_3]$ or $E[V_2,V_3]$. Since, in $G'$, the edge sets $E[v_1,v_3]$ and $E[v_2,v_3]$ are non-empty, $F-C\neq\varnothing$. Thus by Lemma \ref{lem: unbalanced cycles preserved by contraction} there is an unbalanced cycle $C'\subseteq C-F$ in $G'$, a contradiction. Thus $\qbg$ satisfies outcome $(i)$.

Next suppose that $\qbgp$ is the quasi-biased graph of Lemma \ref{lem: rank-2 disconnected loopless quasi-biased graphs}$(ii)$. We will show that $\qbg$ satisfies outcome $(ii)$. By Lemma \ref{lem: at least one frame-type loop}, $F$ contains no cycles in $\cF$, so $G'=G/F$. Since $G'$ has an unbalanced cycle, $F$ contains no cycles in $\cL$. Let $v_1$ and $v_2$ be the vertices of $G'$ with associated vertex partition $V_1,V_2$ of $G$ such that $G[V_i]$ is connected for $i=1,2$. Let $X$ be a maximal balanced set of edges of $G'$ and let $Y=E(G')-X$. Note that $(X,Y)$ is a partition of $E[V_1,V_2]$ in $G$. We now show that $X$ and $Y$ form a pair of balancing sets of $(G,\cB,\cL,\cF)$. Suppose that $G\del X$ contains an unbalanced cycle $C$. By Lemma \ref{lem: unbalanced cycles preserved by contraction}, $C-F$ contains an unbalanced cycle $C'\subseteq Y$ in $G'$, a contradiction. Thus symmetry implies that both $X$ and $Y$ are balancing sets, so $\qbg$ satisfies outcome $(ii)$.

Next suppose that $(G', \cB', \cL', \cF')$ is the quasi-biased graph of Lemma \ref{lem: rank-2 disconnected loopless quasi-biased graphs}$(iii)$. 
Let $v_1$ and $v_2$ be the vertices of $G'$ so that all loops are incident with $v_1$.
Let $X$ be the set of loops and let $Y$ be the set of non-loops of $G'$.
Note that $F$ contains no cycles in $\cL$ because $(G', \cB', \cL', \cF')$ has an unbalanced cycle.
We will consider two subcases depending on whether or not $F$ contains a cycle in $\cF$.
First suppose that $F$ contains no cycles in $\cF$.
We will show that $(G, \cB, \cL, \cF)$ satisfies outcome $(iii)$.
Note that $G[F]$ is balanced, and $G' = G/F$.
Since $|V(G')| = 2$, it follows that $G[F]$ has two components $G_1$ and $G_2$ so that for $i = 1,2$ the vertices of $G_i$ are all identified via contraction to form vertex $v_i$.
For $i = 1,2$ let $V_i = V(G_i)$, and note that $(V_1, V_2)$ is a partition of $V(G)$.
Since $(G', \cB', \cL', \cF') \del X$ is balanced and $G[F]$ is balanced, Lemma \ref{lem: unbalanced cycles preserved by contraction} implies that $(G, \cB, \cL, \cF) \del X$ is balanced.
Clearly $X \subseteq E(G[V_1])$ because $X$ is a set of loops at $v_1$ in $G/F$.
The graph $G[V_1]\del X$ is connected because it has $G_1$ as a spanning subgraph, and $G[V_2]$ is connected because it has $G_2$ has a spanning subgraph.
Therefore $(G, \cB, \cL, \cF)$ satisfies outcome $(iii)$.

In the second subcase suppose that $F$ contains a cycle $C$ in $\cF$.
We will show that $(G, \cB, \cL, \cF)$ satisfies outcome $(iii)'$.
Let $T$ be an edge-maximal forest of $G[F]$.
Note that $G[\cl(T)]$ is balanced, and that each element in $F - \cl(T)$ is a frame-type loop of $(G, \cB, \cL, \cF)/\cl(T)$, by the edge-maximality of $T$.
Let $V'$ be the set of vertices of $G/\cl(T)$ incident with a loop in $F$.
Then $G'$ is obtained from $G/\cl(T)$ by deleting $V'$, and for each edge $e$ with exactly one end in $V'$, replacing $e$ with a frame-type loop at the end of $e$ that is not in $V'$.
In particular, $V(G/\cl(T)) = V' \cup \{v_1, v_2\}$, and $G/\cl(T)$ has no edges between $V'$ and $v_2$ since $G'$ has no loops at $v_2$.
Also note that in $G/\cl(T)$, there are no non-loop edges among vertices of $V'$.
To see this, if $e$ is a non-loop between vertices of $V'$, then it is spanned by the two loops in $F$ at its end-vertices, and therefore $e \in F$, which contradicts that $T$ is an edge-maximal forest of $G[F]$.
Since $\si(G)$, and therefore $\si(G/\cl(T))$, is obtained from a complete graph by deleting at most two edges and there are no edges between $V'$ and $v_2$ in $G/\cl(T)$, we see that $|V'| = 1$; let $v_3 \in V'$.
Therefore $G[\cl(T)]$ has three components $G_1$, $G_2$, and $G_3$ so that for $i = 1,2,3$ the vertices in $G_i$ are identified via contraction of $\cl(T)$ to form vertex $v_i$.
For $i = 1,2,3$ let $V_i = V(G_i)$, and note that $(V_1, V_2, V_3)$ is a partition of $V(G)$ with no edges between $V_2$ and $V_3$.
Clearly $X \subseteq E(G[V_1 \cup V_3])$ because $X$ is a set of loops at $v_1$ in $G'$.
The graphs $G[V_1] \del X$, $G[V_2]$, and $G[V_3] \del X$ are connected because $G[V_i]/\cl(T)\del X$ are connected for each $i=1,2,3$.
Finally, $G[V_3]$ contains a cycle in $\cF$ because $v_3$ has an incident frame-type loop in $C-\cl(T)$ in $G/\cl(T)$.
Therefore $(G, \cB, \cL, \cF)$ satisfies outcome $(iii)'$.

Next suppose that $\qbgp$ is the quasi-biased graph of Lemma \ref{lem: rank-2 disconnected loopless quasi-biased graphs}$(iv)$. We will show that $G$ satisfies outcome $(iv)$. By Lemma \ref{lem: at least one frame-type loop}, $F$ contains no cycles in $\cF$, so $G'=G/F$. And since $G'$ has an unbalanced cycle, $F$ contains no cycles in $\cL$. Let $X$ be the set of loops of $G'$. Let $v_1,v_2$ be the vertices of $G'$ with corresponding partition $(V_1,V_2)$ of $V(G)$ such that $G[V_i]\del X$ is connected for $i=1,2$. It is straightforward to see that $X\subseteq E(G[V_1])\cup E(G[V_2])$. Because $G[F]$ is balanced and $\qbgp\del X$ is balanced, it follows by Lemma \ref{lem: unbalanced cycles preserved by contraction} that $\qbg\del X$ is balanced. Since $G'$ has lift-type loops at each vertex $v_i$, it follows that $G[V_1]$ and $G[V_2]$ each contain cycles in $\cL$. Thus $\qbg$ satisfies outcome $(iv)$.

We now consider the case when $\qbg/F$ is disconnected. Since $G$ is obtained by deleting at most two edges from a complete graph with at least four vertices, $G$ is connected. If $G[F]$ contains no cycles in $\cF$, then $G'=G/F$. But this is impossible, since edge contraction in a connected graph cannot result in a disconnected graph. Thus $G[F]$ contains a cycle $C$ in $\cF$ and by Lemma \ref{lem: contract a cycle}, $M/F$ is a frame matroid. Also, $G[F]$ has no cycles in $\cL$ because $G'$ has unbalanced cycles. Let $T$ be an edge-maximal forest of $F$. We note that $G[\cl(T)]$ is balanced, and each element of $F-\cl(T)$ is a frame-type loop in $\qbg/\cl(T)$ by the edge-maximality of $T$. Let $V'$ be the set of vertices of $G/\cl(T)$ that are incident with a loop in $F$. Then $G'$ is obtained from $G/\cl(T)$ by deleting $V'$ and, for each edge $e$ with exactly one end in $V'$, replacing $e$ with a frame-type loop at the end of $e$ that is not in $V'$. 
Moreover, we see that in $G/\cl(T)$, there are no non-loop edges between vertices of $V'$. Otherwise, they are spanned by unbalanced loops at both end-vertices and belong to $F$, contradicting the edge-maximality of $T$. 
Since $G'$ is unbalanced and $r(M/F)=2$, it follows that $\lvert V(G')\rvert=2$. Let $\{v_1,v_2\}$ be the vertices of $G'$. We see that $V(G/\cl(T))=V'\cup\{v_1,v_2\}$. Since $G'$ is disconnected, there are no edges between $v_1$ and $v_2$. Since $\si(G/\cl(T))$ is obtained by deleting a path of length at most two from a complete graph, and since there are no edges between $v_1$ and $v_2$, we see that $\lvert V'\rvert=1$. Let $v_3\in V'$. Then $G[\cl(T)]$ has three components $G_1,G_2$ and $G_3$ so that for $i=1,2,3$ the vertices in $G_i$ are identified via contraction to form vertex $v_i$. For each $i$, let $V_i=V(G_i)$, and note that $(V_1,V_2,V_3)$ is a partition of $V(G)$ with no edges between $V_1$ and $V_2$. Since each graph $G[V_i]/\cl(T)$ is connected it follows that $G[V_i]$ is connected for $i=1,2,3$. Moreover, $G[V_3]$ contains a cycle in $\cF$ because $v_3$ is incident with a frame-type loop in $C-\cl(T)$ in $G/\cl(T)$.
Thus $\qbg$ satisfies outcome $(v)$.  
\end{proof}

We have now characterized all unbreakable quasi-graphic matroids on a graph $G$ with at least four vertices so that $\si(G)$ is the complement of a path with at most two edges.
It is natural to wonder which of these quasi-graphic matroids are $3$-connected.
Our next result answers this question.

\begin{proposition} \label{prop: which ones are 3-connected?}
Let $M = M(G, \cB, \cL, \cF)$ be a simple quasi-graphic matroid so that $\si(G)$ is obtained from a complete graph by deleting a path with at most two edges, and $|V(G)| > 6$.
Then $M$ is not $3$-connected if and only if $(G, \cB, \cL, \cF)$ has a balancing set $X$ of edges with $r_M(X) \le 2$.
\end{proposition}
\begin{proof}
It is straightforward to see that if $X$ is a balancing set with $r_M(X)\leq2$, then $M$ is not 3-connected. To show this, let $Y=E(M)-X$. Note that, since $X$ is a balancing set, $r(Y)\leq\lvert V(G)\rvert-1=r(M)-1$. 
Let $r(X) = i$, so $|X|, |Y| \ge i$.
Then
\[
r(X)+r(Y)-r(M)\leq i+r(M)-1-r(M)\leq i - 1,
\]
implying that $M$ is not 3-connected.

Conversely, suppose that $M$ is not $3$-connected. If all cycles are balanced, then $\si(G) = G$ and $G$ is $3$-connected and it follows from \cite[Proposition 8.1.9]{Oxley2011} that $M$ is $3$-connected. So not all cycles are balanced, so $r(M) = |V(G)| > 6$.
Then $E(M)$ has a partition $(X, Y)$ so that $r(X) + r(Y) - r(M) \le i$ and $|X|, |Y| \ge i+1$ for some $i \in \{0,1\}$. We begin by showing the following.

\begin{claim}\label{$G[X]$ or $G[Y]$ contains a spanning tree of $G$}
$G[X]$ or $G[Y]$ contains a spanning tree of $G$.
\end{claim}
\begin{proof}
Suppose $G[X]$ does not contain a spanning tree of $G$. We will show that with $G[Y]$ contains a spanning tree of $G$, or $G[X]$ contains a forest with at least $|V(G)| - 2$ edges and $G[Y]$ contains a forest with at least $|V(G)| - 3$ edges. Since $G[X]$ does not contain a spanning tree, it is disconnected and we will consider cases according to the number of components in $G[X]$. First, assume that $G[X]$ has $4$ components. We note that the argument is identical for $G[X]$ having $c\geq4$ components. Then $G[Y]$ contains $K_{n_1,n_2,n_3,n_4}$ for some positive integers $n_1,n_2,n_3,n_4$ such that $n_1+n_2+n_3+n_4=\lvert V(G)\rvert$. Observe that $K_{n_1,n_2,n_3,n_4}$ is connected and does not have any $k$-edge cuts for $k\leq2$. Since $\si(G)$ is obtained by deleting a path of length at most 2 from a complete graph, $G[Y]$ contains a spanning tree. Next we assume that $G[X]$ has $3$ components. Then $G[Y]$ contains $K_{n_1,n_2,n_3}$ for integers $n_1,n_2,n_3$ so that $n_1+n_2+n_3=\lvert V(G)\rvert$. 
Note that $G[Y]$ has no cut-edge. So $G[Y]$ has a $(|V(G)|-2$)-edge forest, and $G[X]$ has a $(|V(G)|-3)$-edge forest. Finally assume that $G[X]$ has $2$ components. Thus $G[Y]$ contains $K_{n_1,n_2}$ for integers $n_1,n_2$ so that $n_1+n_2=\lvert V(G)\rvert$. Then $G[Y]$ contains a forest with $|V(G)|-3$ edges. 

In all cases, by possibly switching $X$ and $Y$, we may assume that $G[X]$ contains a forest with at least $\lvert V(G)\rvert-2$ edges and $G[Y]$ contains a forest with at least $\lvert V(G)\rvert-3$ edges. Thus $r(X) \ge r(M) - 2$ and $r(Y) \ge r(M) - 3$. It follows that, 
\[i \geq  r(X)+r(Y)-r(M) \geq r(M)-2+r(M)-3-r(M) \geq r(M)-5\]
which implies that $r(M) \le 5 + i \le 6$, a contradiction. This concludes the proof of \ref{$G[X]$ or $G[Y]$ contains a spanning tree of $G$}.
\end{proof}

So we may assume that $G[Y]$ contains a spanning tree of $G$. If $r(Y) = r(M)$, then $r(X) \le i \leq 1$ and $|X| \ge i + 1$ which implies that $X$ has a loop or a parallel pair, which contradicts our assumption.
Otherwise $r(Y) = r(M) - 1$, which implies that $G[Y]$ is balanced and $r_M(X)\leq i+1\leq 2$, hence $X$ is a balancing set with $r_M(X)\leq2$. 
\end{proof}

Theorem \ref{thm: structure of breakable complete quasi-graphic matroids} and Proposition \ref{prop: which ones are 3-connected?} combine to prove Theorem \ref{thm: main converse for nearly complete graphs}.

\subsection{Cycles}\label{sec: cycles}

In this section we will characterize the unbreakable $3$-connected lifted-graphic matroids $M = M(G, \cB, \cL, \varnothing)$ for which $\si(G)$ is a cycle, proving Theorem \ref{thm: main converse for cycles}.
We first characterize $3$-connectivity.

\begin{proposition} \label{lem: at most one edge of $M$ is not in an unbalanced $2$-cycle}
    Let $M=M(G, \cB, \cL, \varnothing)$ such that $\si(G)$ is a cycle and $|V(G)| > 3$. Then $M$ is $3$-connected if and only if $G$ does not have balanced loops or $2$-cycles, $G$ has at most one unbalanced loop, and at most one edge of $G$ is not in an unbalanced $2$-cycle.
\end{proposition}
\begin{proof}
    Suppose $M$ is 3-connected. Then $G$ has no balanced loops or 2-cycles. If $G$ has at least two unbalanced loops, these loops are a parallel pair of $M$, a contradiction. Suppose by contradiction that $G$ has a pair of edges $\{e,f\}$, neither of which belongs to an unbalanced 2-cycle. Letting $X=\{e,f\}$ and $Y=E(G)-\{e,f\}$, we note that $\lvert Y\rvert\geq2$ since $\lvert V(G)\rvert\geq4$. Moreover, $r(X)=2$ and, since $G$ is unbalanced, $r(Y)\leq r(M)-1$ by Proposition \ref{prop: rank of X} because $G[Y]$ has more components than $G$. Thus $r(X)+r(Y)-r(M)\leq1$, contradicting that $M$ is 3-connected. 

    Conversely, suppose that $G$ has no balanced loops or $2$-cycles, $G$ has at most one unbalanced loop, and at most one edge of $\si(G)$ is not in an unbalanced $2$-cycle of $G$. Recall that for lifted-graphic matroids, for a subset $A\subseteq E(M)$, $l(A)=1$ if $A$ contains an unbalanced cycle and $l(A)=0$ otherwise. Suppose that $E(M)$ has a partition $(X,Y)$ so that $\lvert X\rvert,\lvert Y\rvert\geq i+1$ and $r(X)+r(Y)-r(M)\leq i\leq1$ for $i\in\{0,1\}$. First assume that $l(X)=l(Y)=1$. Note that $\lvert V(G)\rvert=\lvert V(G[X])\rvert+\lvert V(G[Y])\rvert-\lvert V(G[X])\cap V(G[Y])\rvert$ and $\lvert V(G[X])\cap V(G[Y])\rvert \geq c(X)+c(Y)$ because $\si(G)$ is a cycle. Then, 
    \begin{align*}
        r(X)+r(Y)&=\lvert V(G[X])\rvert-c(X)+l(X)+\rvert V(G[Y])\rvert-c(Y)+l(Y)\\
        &=\lvert V(G)\rvert + \lvert V(G[X])\cap V(G[Y])\rvert-(c(X)+c(Y))+2\\
        &\geq r(M)+2,
    \end{align*}
    a contradiction.

   Now suppose $l(X)=0$. We claim that $G[Y]$ contains a spanning tree of $G$. To see this, note that since $l(X)=0$, for any non-loop $e\in X$, if $e$ belongs to an unbalanced 2-cycle with an edge $f$, then $f\notin X$. Since $G$ has at most one non-loop not in an unbalanced 2-cycle, this means that $Y$ contains at least one edge from each unbalanced 2-cycle of $G$, which forms a spanning tree of $G$. Thus $r(Y)\geq\lvert V(G)\rvert-1=r(M)-1$. If $r(Y)=r(M)$, we see that $r(X)\leq i$ since $r(X)+r(Y)-r(M)\leq i$. For $i=0,1$, this implies that $X$ contains a loop or parallel pair since $|X| \ge i + 1$. This contradicts that $G$ has no balanced loops or balanced 2-cycles, and at most one unbalanced loop. Next, if $r(Y)=r(M)-1$, then $r(X)\leq i+1 \leq 2$. Since $G[Y]$ contains a spanning tree of $G$, this also implies that $l(Y)=0$. By symmetry between $X$ and $Y$, we see that $G[X]$ also contains a spanning tree. 
   Then $r(M) \le r(X) + 1 \le 3$, so $\lvert V(G)\rvert\leq3$, a contradiction. 
\end{proof}

We next show that if $\si(G)$ is a cycle and $M(G, \cB, \cL, \varnothing)$ is $3$-connected, then it is unbreakable, completing the proof of Theorem \ref{thm: main converse for cycles}.

\begin{proposition} \label{prop: connected lifted graphic si(G) is a cycle}
If $M = M(G, \cB, \cL, \varnothing)$ is $3$-connected and $\si(G)$ is a cycle and $|V(G)| > 3$, then $M$ is unbreakable.
\end{proposition}
\begin{proof}
  By Theorem \ref{thm: the unbreakable graphic matroids}, we may assume that $\cL\neq\varnothing$. Since $M(G,\cB,\cL,\varnothing)$ is $3$-connected, by Proposition \ref{lem: at most one edge of $M$ is not in an unbalanced $2$-cycle} we see that $G$ has at most one loop, and all $2$-cycles are in $\cL$. If $M$ is not unbreakable, then by Lemma \ref{lem: contract to rank 2} there exists a flat $F$ of $M$ so that $r(M/F)=2$ and $M/F$ is disconnected. Note that $M/F$ is a lifted-graphic matroid with $G/F$ connected. From Lemma \ref{lem: rank-2 disconnected loopless quasi-biased graphs}, outcomes $(ii),(iii)$ and $(iv)$ are the only possible structures for $G/F$. Outcome $(i)$ clearly cannot happen since $\si(G/F)$ is a cycle or a single edge, because $\si(G)$ is a cycle. 

  If $F$ contains a cycle in $\cL$, then $M/F$ is graphic. In this case, $\si(G/F)$ is a cycle on $3$ vertices because $r(M/F) = 2$ and $\si(G)$ is a cycle.
  Note that $G/F$ is loopless because $F$ is a flat.
  It then follows from Theorem \ref{thm: the unbreakable graphic matroids} that $M/F$ is connected, a contradiction. 
  So $F$ does not contain a cycle in $\cL$. 
  Then $F$ is a forest on $|V(G)|-2$ edges and $F$ has either $1$ or $2$ components. To see this, note that since all 1- and 2-cycles of $G$ are unbalanced and $\si(G)$ is a cycle, the only balanced cycles are Hamiltonian. Thus if $F$ contains a cycle it follows that $r(M/F)\leq1$, so $M/F$ is connected, a contradiction. Also, $F$ cannot have less than $|V(G)|-2$ edges since $r(M/F)=2$. From this we see that $F$ has at most $2$ components. 

  Proposition \ref{lem: at most one edge of $M$ is not in an unbalanced $2$-cycle} implies that an edge of $F$ is contained in an unbalanced 2-cycle $C$. From this we see that $G/F$ is a graph on two vertices with a loop on at least one of them, from the edge of $C$ that was not in $F$, which excludes outcome $(ii)$ of Lemma \ref{lem: rank-2 disconnected loopless quasi-biased graphs}. If $E(M)-F$ also contains a $2$-cycle in $\cL$ whose edges are not parallel in $G$ with any edges of $F$, then in the graph of $M/F$ this $2$-cycle remains, which cannot happen in outcomes $(iii)$ or $(iv)$ of Lemma \ref{lem: rank-2 disconnected loopless quasi-biased graphs}. Since $\lvert F\rvert=\lvert V(G)\rvert-2$, it follows that $G$ has at least two non-loops contained in $E(M)-F$ which are not in unbalanced 2-cycles. By Proposition \ref{lem: at most one edge of $M$ is not in an unbalanced $2$-cycle}, this contradicts that $M$ is 3-connected. We conclude that $M$ is unbreakable.
\end{proof}

\section{Beyond 3-connectivity} \label{sec: future work}

In this section we will prove Theorem \ref{thm: non-degen unbreakability} and Theorem \ref{thm: connected not 3-connected}. For Theorem \ref{thm: 3-connected quasi matroid with G 2 connected} and Theorem \ref{thm: 4.1}, $M$ is assumed to be 3-connected, and this assumption provides the constraints that characterize the underlying graphs. When neither of the collections $\cL$ and $\cF$ are degenerate, the fact that $M$ has a proper tripartition provides the necessary ground for similar results.

This section will utilize the notion of a \emph{link-sum} from \cite{Bowler-Funk-Slilaty20} which explicitly verifies that the class of quasi-graphic matroids is closed under 2-sums with graphic matroids. Let $G_1$ and $G_2$ be graphs and let $(\cB,\cL,\cF)$ be a proper tripartition of the cycles of $G_1$. Assume that $E(G_1)\cap E(G_2)=\{e\}$ and that $e$ is a non-loop in both $G_1$ and $G_2$. The 2-sum of the matroid $M(G_1,\cB,\cL,\cF)$ and the cycle matroid $M(G_2)$ on basepoint $e$ may be realized in the graphs as follows. Let $u_1,v_1$ be the endpoints of $e$ in $G_1$ and let $u_2,v_2$ be the endpoints of $e$ in $G_2$. The \emph{link-sum} of $G_1$ and $G_2$ on $e$ is the graph, denoted $G_1\oplus_2^e G_2,$ obtained by identifying $u_1$ with $u_2$ and identifying $v_1$ with $v_2$, then deleting $e$, together with the following tripartition of its cycles. For a collection of cycles $\cC$ and an edge $e$ write $\cC-e$ for the set $\{C\in\cC: e \notin C\}$. Let $\cB'$ be the union of $\cB-e$, the set of all cycles of $G_2-e$ and the set 
$$
\{P\cup Q: P \text{ is a } (u_1,v_1) \text{-path in $G_1$ with } P\cup e\in\cB \text{ and $Q$ is a } (u_2,v_2)\text{-path in $G_2$}\}. 
$$
Let $\cL'$ be the union of $\cL-e$ and the set
$$
\{P\cup Q: P \text{ is a } (u_1,v_1) \text{-path in $G_1$ with } P\cup e\in\cL \text{ and $Q$ is a } (u_2,v_2)\text{-path in $G_2$}\} 
$$
and let $\cF'$ be the union of $\cF-e$ and the set 
$$
\{P\cup Q: P \text{ is a } (u_1,v_1) \text{-path in $G_1$ with } P\cup e\in\cF \text{ and $Q$ is a } (u_2,v_2)\text{-path in $G_2$}\}. 
$$
Then $(\cB',\cL',\cF')$ is a proper tripartition of the cycles of $G_1\oplus_2^e G_2$, and the 2-sum of $M(G_1,\cB,\cL,\cF)$ and $M(G_2)$ on $e$ is equal to $M(G_1\oplus_2^e G_2, \cB',\cL',\cF')$ \cite[Section 4]{Bowler-Funk-Slilaty20}.

Note that if $(G, \cB, \cL, \cF)$ is a quasi-biased graph and $(A, B)$ is a partition of $E(G)$ so that $G[A]$ and $G[B]$ share two vertices and $G[B]$ is balanced, then $(G, \cB, \cL, \cF)$ decomposes as a link-sum $G_1 \oplus_2^e G_2$ with $G_1 \del e = G[A]$ and $G_2 \del e = G[B]$.
Moreover, if $(\cB', \cL', \cF')$ is the corresponding proper tripartition of the cycles of $G_1$, then it is straightforward to check that $\cL'$ (resp. $\cF'$) is degenerate if and only if $\cL$ (resp. $\cF$) is degenerate.

With the link-sum construction in mind, we have the following theorem of Bowler, Funk, and Slilaty \cite{Bowler-Funk-Slilaty20}, which shows that when $M=M\qbg$ is connected but not 3-connected, and has non-degenerate collections $\cL$ and $\cF$, then $M$ is obtained through a very specific 2-sum decomposition.

\begin{theorem}[{\cite[Theorem 4.8]{Bowler-Funk-Slilaty20}}]\label{thm: non-degen decomposition}
    Let $M$ be a connected matroid of the form $M(G,\cB,\cL,\cF)$ with $(\cB,\cL,\cF)$ a proper tripartition of $G$ such that neither $\cL$ nor $\cF$ is degenerate. If $M$ is not 3-connected, then $M$ is obtained via 2-sums of graphic matroids and a single 3-connected matroid of the form $M(H,\cB',\cL',\cF')$, for some graph $H$ with proper tripartition $(\cB',\cL',\cF')$.
\end{theorem}

We will need a strengthening of this result in which we show that $\cL'$ and $\cF'$ are both non-degenerate.
The key is the following lemma. We say a 2-sum (or link-sum) $M=M_1\oplus_2M_2$ is \emph{trivial} if $M\cong M_i$ for some $i\in\{1,2\}$.

\begin{lemma} \label{lem: doubly non-degenerate implies 2-sum with graphic}
Let $(G, \cB, \cL, \cF)$ be a quasi-biased graph such that $\cL$ and $\cF$ are non-degenerate, and let $M = M(G, \cB, \cL, \cF)$. If $M$ is connected but not $3$-connected, then $M = M(G', \cB', \cL', \cF') \oplus_2 M_2$ for a nontrivial 2-sum with $\cL'$ and $\cF'$ both non-degenerate, $M_2$ graphic, and $M(G', \cB', \cL', \cF')$ connected.
\end{lemma}

\begin{proof}
To prove this, we will follow the proof as \cite[Theorem 4.8]{Bowler-Funk-Slilaty20}, but without the induction on the number of graphic matroids in the 2-sum. 

By \cite[Theorem 4.6]{Bowler-Funk-Slilaty20}, $G$ is $2$-connected. Since $M$ is connected but not $3$-connected, it must have a $2$-separation $(A,B)$. Let $c_A$ be the number of components of $G[A]$ and $c_B$ the number of components of $G[B]$. Choose $(A,B)$ so as to minimize $c_A + c_B$. Let $S$ be the set of vertices incident to edges in both $A$ and $B$. Since $G$ is $2$-connected, no component of $A$ or $B$ can meet $S$ in fewer than $2$ vertices.

\begin{claim}
    There is a balanced component of $A$ or $B$ having precisely $2$ vertices in $S$.
\end{claim} 
\begin{proof}
Suppose that the claim is false.
We use $r_G$ to denote the rank function of $M(G)$ ($r$ is used for the rank function of $M$). We define $d$ to be $(r(A) - r_G(A)) + (r(B) - r_G(B))$. Then the equation for the order of $(A,B)$ tells us that
\[
1 = r(A) + r(B) - r(E(G)) = (r_G(A) + r_G(B) - r_G(E(G))) - 1 + d \ge d,
\]
since $G$ is $2$-connected. So one of $r(A) - r_G(A)$ or $r(B) - r_G(B)$, without loss of generality the second, must be zero. Thus $B$ is balanced. Since every component of $A$ meets $S$ in at least $2$ vertices we have $c_A \le |S|/2$. 
If some component of $B$ meets $S$ in $2$ vertices, then the statement holds, so we may assume that every component of $B$ meets $S$ in at least $3$ vertices,  $c_B \le |S|/3$. Using the above formula for the order of $(A,B)$ once more, we have
\begin{align*}
1 &= (r_G(A) + r_G(B) - r_G(E(G))) - 1 + d \\
&= |V(A)| - c_A + |V(B)| - c_B - |V(G)| + d \\ 
&= |S| - c_A - c_B + d \\
&\ge |S| - |S|/2 - |S|/3 + d \\
&= |S|/6 + d \\
&> d.
\end{align*}
So $d = 0$ and thus $A$ is also balanced and by the above claim we have $c_A \le |S|/3$. Substituting this into the above calculation we have
\[
1 = |S| - c_A - c_B \ge |S| - |S|/3 - |S|/3 = |S|/3,
\]
so $S$ contains at most $3$ vertices. Since both $A$ and $B$ are balanced, every cycle in $\mathcal{F}$ must meet both $A$ and $B$ and so must contain at least $2$ of these $3$ vertices, contradicting the assumption that $\mathcal{F}$ is non-degenerate.
\end{proof}

Let $X$ be a balanced component of $A$ having precisely two vertices in $S$. Then $X$ consists of a single edge $x$ (else $M$ would be a $2$-sum of a connected quasi-graphic matroid with non-degenerate sets of frame-type and lift-type cycles and a graphic matroid, given by the link-sum corresponding to the $2$-separation $(X, E(G) - X)$). But now $r(A - x) = r(A) - 1$ and $r(B \cup x) \le r(B) + 1$, so that the order of $(A - x, B \cup x)$ is at most that of $(A,B)$. Since the order of $(A,B)$ is at most $1$ and that of $(A - x, B \cup x)$ is at least $1$, we must have $r(B \cup x) = r(B) + 1$ and the order of $(A - x, B \cup x)$ is precisely $1$.

Since we chose $(A,B)$ to minimize $c_A + c_B$, $(A - x, B \cup x)$ cannot be a $2$-separation and so $|A| = 2$. Since $M$ is connected and neither $\mathcal{L}$ nor $\mathcal{F}$ is degenerate, $G$ has no loops. Thus both elements of $A$ are non-loops, so $|S| = 4$. Since $r(B \cup x) = r(B) + 1$, $B$ is not spanning and so $A$ is codependent. Since $M$ is connected, $A$ must be a cocircuit, so its two elements are in series in $M$. But then $M$ is a $2$-sum of $M/x$ with $M(K_3)$. Note that $M/x$ is connected since $M$ is, and that $M/x = M((G,\cB,\mathcal{L},\mathcal{F})/x)$. 
We will write $(G', \cB/x, \cL/x, \cF/x)$ for $(G, \cB, \cL, \cF)/x$.
Our next aim is to show that neither $\mathcal{L}/x$ nor $\mathcal{F}/x$ is degenerate; then the lemma statement will hold with $M_2 = M(K_3)$.

There are two cases. The first is that at least two components of $B$ meet $X$. In this case, since each of these components meets $S$ in at least $2$ vertices and $S$ has only $4$ vertices, these two are the only components of $B$. Let $v$ be an endpoint of $x$. Since $\mathcal{F}$ is non-degenerate, there is some $C \in \mathcal{F}$ not containing $v$. Thus $C$ must be contained in some component $Y_1$ of $B$. Let $Y_2$ be the other component. Then $Y_2$ cannot also contain a cycle in $\mathcal{F}$, since $r(B) = r(M) - 1 = |V(B)| - 1$. But $Y_2$ cannot contain a cycle in $\mathcal{L}$ either, since $(B,\mathcal{L},\mathcal{F})$ is proper. So $Y_2$ is balanced, and has precisely $2$ vertices in $S$, so as above it consists of a single edge $y$. But then $x$ and $y$ are in series in $G$, and so neither $\mathcal{L}/x$ nor $\mathcal{F}/x$ can be degenerate.

The second case is that only one component $Y$ of $B$ meets $X$. Then $Y$ must contain at least one further vertex of $S$, and since $S$ has only $4$ vertices $Y$ must contain $S$ and so in fact $B$ must be connected. In this case, since $r(B) = r(M) - 1 = |V(B)| - 1$ we must have that $B$ is balanced. Since $\mathcal{F}$ is not degenerate, there are disjoint cycles $C_1$ and $C_2$ in $\mathcal{F}$. Then $C_1$ and $C_2$ must meet $A$ in different edges, so one of them, say $C_1$, must contain $x$. Then $C_1/x$ and $C_2$ are disjoint cycles in $\mathcal{F}/x$, which is therefore also non-degenerate. Similarly $\mathcal{L}/x$ is non-degenerate.
\end{proof}

We can now prove the following strengthening of Theorem \ref{thm: non-degen decomposition} by inductively applying Lemma \ref{lem: doubly non-degenerate implies 2-sum with graphic}.

\begin{theorem}\label{thm: non-degen decomposition stronger}
    Let $M$ be a connected matroid of the form $M(G,\cB,\cL,\cF)$ with $(\cB,\cL,\cF)$ a proper tripartition of $G$ such that neither $\cL$ nor $\cF$ is degenerate. Then $M$ is obtained via 2-sums of graphic matroids and a single 3-connected matroid of the form $M(H,\cB',\cL',\cF')$, for some graph $H$ with proper tripartition $(\cB',\cL',\cF')$ with $\cL'$ and $\cF'$ both non-degenerate.
\end{theorem}
\begin{proof}
We use the induction part of the proof of \cite[Theorem 4.8]{Bowler-Funk-Slilaty20}. Suppose that the statement is false and let $M$ be a counterexample whose number of elements is minimal. If $M$ is $3$-connected, the statement trivially holds true with $(H', \cB', \cL', \cF') = (G, \cB, \cL, \cF)$. If $M$ is not $3$-connected, we apply Lemma \ref{lem: doubly non-degenerate implies 2-sum with graphic}. Thus there is a nontrivial 2-sum of matroids $M'=M\qbgp$ and $M_2$ so that $\cL'$ and $\cF'$ are non-degenerate and $M_2$ is graphic. By induction, $M'$ decomposes as 2-sums of a single 3-connected quasi-graphic matroid $M''=M(G'',\cB'',\cL'',\cF'')$ and graphic matroids so that $\cL''$ and $\cF''$ are both non-degenerate. Thus so does $M$, a contradiction.
\end{proof}

The previous results tell us that for those quasi-graphic matroids $M$ that are not 3-connected, the non-degeneracy of $\cL$ and $\cF$ restricts the possible structure of $M$. The following critical results of Oxley and Pfeil \cite{Oxley-Pfeil2022} provide additional constraints when $M$ is unbreakable and not 3-connected. A \emph{free element} of a matroid $M$ is an element $e$ which is not a coloop so that the only circuits containing $e$ are spanning. Thus a loopless matroid with a free element is connected.
\begin{lemma}[{\cite[Lemma 3.5]{Oxley-Pfeil2022}}] \label{lem: free element implies unbreakable}
    If a loopless matroid $M$ has a free element, then $M$ is unbreakable.
\end{lemma}

\begin{lemma}[{\cite[Lemma 3.6]{Oxley-Pfeil2022}}] \label{lem: 2-sum is unbreakable iff free element}
    Let $M_1$ and $M_2$ be loopless matroids with $E(M_1) \cap E(M_2)=\{p\}$ and $r(M_i)\geq 2$ for each $i$. The 2-sum of $M_1$ and $M_2$ with respect to the basepoint $p$ is unbreakable if and only if $p$ is free in both $M_1$ and $M_2$.
\end{lemma}

To utilize these results, we prove a straightforward lemma about graphic matroids as components of 2-sums with free element basepoints. Note that this result applies to $G$ rather than just $\si(G)$.   
\begin{lemma}\label{lem: link-sum graphic matroids are cycles}
    Suppose that $M$ is a simple quasi-graphic matroid obtained via a non-trivial 2-sum over basepoint $p$ with a graphic matroid $M(G)$ such that $p$ is a free element of $M(G)$. Then $G$ is a cycle on at least 3 vertices. 
\end{lemma}
\begin{proof}
    Suppose that $M=M(G)\oplus_2M'$. Since $M$ is simple, $G$ does not have any loops other than possibly $p$. Since $p$ is a free element, it is not a loop. Thus Lemma \ref{lem: free element implies unbreakable} implies that $M(G)$ is unbreakable. Theorem \ref{thm: the unbreakable graphic matroids} implies that $\si(G)$ is either a complete graph or a cycle. We see that $\si(G)$ is not a complete graph on more than 3 vertices since $p$ is a free element and a non-loop edge of $G$. Thus $\si(G)$ is an edge or a cycle. Suppose $\si(G)$ is an edge. If $\lvert E(G)\rvert\geq3$, this implies that $M$ has a parallel pair, contradicting that $M$ is simple. If $\lvert E(G)\rvert=2$, then $M\cong M'$, contradicting that the 2-sum is non-trivial. If $G$ consists of a single edge, it must be the edge $p$, in which case $p$ is a coloop of $M(G)$, contradicting that $p$ is a free element. Thus we may assume that $\si(G)$ is a cycle on at least 3 vertices.

    Since $M$ is simple, $G$ may not have any 2-cycles that do not use $p$. If $p$ belongs to a 2-cycle $C$, since $G$ has at least 3 vertices, $C$ does not span $M(G)$, contradicting that $p$ is a free element. Thus $G$ has no 2-cycles or loops, implying that $G$ is a cycle.
\end{proof}

Finally, we need a straightforward observation about the consequences of non-degenerate collections $\cL$ and $\cF$ on an underlying graph $G$ that simplifies to a cycle.

\begin{lemma}\label{lem: nondegen vertex bounds}
    Let $\qbg$ be a quasi-biased graph so that $\cL$ and $\cF$ are both non-degenerate. Then $\lvert V(G)\rvert\geq4$. In addition, when $\si(G)$ is a cycle, $\lvert V(G)\rvert=4$.
\end{lemma}
\begin{proof}
    Since $\cL$ is non-degenerate, there is a pair of vertex-disjoint lift-type cycles $C_1$ and $C_2$. Neither of these cycles may be loops of $G$ since $\cF$ is non-degenerate. Thus $\lvert V(C_i)\rvert\geq2$ for $i=1,2$, and it follows that $\lvert V(G)\rvert\geq4$. Now suppose that $\si(G)$ is a cycle. Note that the only cycles of $G$ are Hamiltonian or 2-cycles, so both $C_1$ and $C_2$ are 2-cycles. Since $\cF$ is non-degenerate, there is a pair of vertex-disjoint frame-type cycles $C$ and $C'$, which must also be 2-cycles of $G$. Moreover, since $G$ has a proper tripartition $(\cB,\cL,\cF)$, the cycle $C$ must meet $C_1$ and $C_2$. This implies that $C_1$ and $C_2$ have a pair of vertices $v_1$ and $v_2$, respectively, which are adjacent via an edge of $C$. Moreover, the cycle $C'$ must meet both $C_1$ and $C_2$. Since it is vertex-disjoint from $C$, we see that the other vertices of $C_1$ and $C_2$, say $v_1'$ and $v_2'$ respectively, must be adjacent via an edge of $C'$. Since $\si(G)$ is a cycle, this implies that $\lvert V(G)\rvert\leq4$, and the conclusion follows.
\end{proof}

Combining the previous results with Theorem \ref{thm: non-degen decomposition stronger}, we apply Theorems \ref{thm: the unbreakable graphic matroids} and \ref{thm: main} to the matroids in the 2-sum for further structural insight. We now prove Theorem \ref{thm: non-degen unbreakability}, which we restate for convenience.

\begin{theorem}
    Let $M=M\qbg$ be a simple connected quasi-graphic matroid such that neither $\cL$ nor $\cF$ is degenerate. If $M$ is not 3-connected, then $M$ is not unbreakable.
\end{theorem}
\begin{proof}
    Suppose for a contradiction that $M$ is unbreakable. We apply Theorem \ref{thm: non-degen decomposition stronger}, beginning the proof by carefully articulating the 2-sum structure. We may choose a 3-connected quasi-graphic matroid $M'=M(H_0,\cB_0,\cL_0,\cF_0)$ such that $\cL_0$ and $\cF_0$ are non-degenerate, and a minimal sequence of connected graphic matroids $M_1,\ldots,M_k$ so that, 
  
    $$
    M\cong((((M'\oplus_{p_1} M_1)\oplus_{p_2} M_2)\ldots)\oplus_{p_k} M_k).
    $$
    Let $N_0 = M'$, and for $i = 1,2,\dots, k$, let 
    $$N_i =((((M'\oplus_{p_1} M_1)\oplus_{p_2} M_2)\ldots)\oplus_{p_i} M_i).$$

    Note that $N_k \cong M$. Moreover, the minimality of $k$ implies that each of these 2-sums are non-trivial.
    Since $M$ is connected, \cite[Proposition 7.1.22]{Oxley2011} implies that $N_i$ is connected for all $i \in \{0,1,\dots, k\}$ and $M_i$ is connected for all $i \in \{1,2,\dots, k\}$.

    \begin{subtheorem}\label{sub: N_i decomp}
        For each $i \in \{0,1,\dots, k\}$, $N_i$ is connected, simple, and unbreakable.
    \end{subtheorem}
    \begin{proof}
        We proceed by induction on $k-i$.  Since $N_k\cong M$ the statement is clearly true for $i=k$.
        Next, let us assume that the matroid $N_i = N_{i-1}\oplus_{p_i} M_i$ is connected, simple, and unbreakable. Lemma \ref{lem: nondegen vertex bounds} implies that $r(N_{i-1})\geq4$. Since $N_{i}$ is simple, the fact that each 2-sum is non-trivial implies that $r(M_{i})\geq2$. Thus it follows from Lemma \ref{lem: 2-sum is unbreakable iff free element} that $p_i$ is a free element of $N_{i-1}$ and $M_i$. Since $N_{i-1}$ is connected it is loopless. Since $p_i$ is a non-loop element of $M_i$ and $N_i$ is simple, $M_i$ is loopless, so Lemma \ref{lem: free element implies unbreakable} implies that $N_{i-1}$ and $M_i$ are unbreakable. Finally, to show that $N_{i-1}$ is simple, we observe again that any parallel pairs of $N_{i-1}$ may not be disjoint from $p_i$, otherwise this contradicts that $N_i$ is simple. But if $\{p_{i-1},e\}$ is a parallel pair of $N_{i-1}$, since $r(N_i)\geq 4$, this parallel pair is not spanning in $N_{i-1}$, contradicting that $p_i$ is a free element. This completes the proof of \ref{sub: N_i decomp}.
    \end{proof}
    
    We now focus on the matroid $N_1=M'\oplus_2M_1$. By \ref{sub: N_i decomp}, the matroids $N_1$ and $N_0=M'$ are connected, simple, and unbreakable. 
    Since $N_1$ is simple and the $2$-sum is non-trivial we see that $r(M_1) \ge 2$.
    Let $G_1$ be a graph for $M_1$ and let $H_1$ be the link-sum $H_0\oplus_2^{p_1}G_1$. Note that $N_1\cong M(H_0\oplus_2^{p_1}G_1,\cB',\cL',\cF')\oplus_{p_1}M_1$ for a proper tripartition $(\cB',\cL',\cF')$ of $H_1$. By Proposition \ref{lem: nondegen vertex bounds}, we have that $\lvert V(H_0)\rvert\geq4$, so $r(M')\geq4$. Since $N_1$ is simple, $M'$ and $M_1$ are loopless.
    By Lemma \ref{lem: 2-sum is unbreakable iff free element}, the basepoint $p_i$ is a free element of both $M'$ and $M_1$ since $N_1$ is unbreakable. Since $M_1$ is connected and $\cL_0$ and $\cF_0$ are non-degenerate, the free element basepoint $p_1$ is a non-loop edge of both $G_1$ and $H_0$. Lemma \ref{lem: link-sum graphic matroids are cycles} implies that $G_1$ is a cycle on at least 3 vertices. In particular, this implies that $\lvert V(H_1)\rvert\geq5$. We also make the straightforward observation that the basepoint $p_1$ does not belong to a lift bracelet of $H_1$. This follows from Proposition \ref{prop: rank of X} and because $p_1$ is a free element of $M'$.

    Since $\cL_0$ and $\cF_0$ are both non-degenerate, and $p_1$ does not belong to a lift bracelet, we see that $H_0$ contains a pair of vertex-disjoint lift-type cycles $C$ and $C'$ that avoid $p_1$. Thus $C$ and $C'$ are lift-type cycles of $H_1$. Moreover, since $C$ and $C'$ are vertex-disjoint, neither of these cycles is Hamiltonian. Contracting $\cl(C)$ from $M'\oplus_{p_1}M_1$, the resulting graph (which is $H_1/\cl(C)$) has no edges from any vertices of $G_1\del\cl(C)$ to any vertices of $H_1/\cl(C)$. Since $C$ is not Hamiltonian and $\cL_0$ contains no loops of $H_0$, each of these vertex sets contain a vertex that is distinct from the end-vertices of $p_1$. Since $N_1/\cl(C)$ is graphic and unbreakable, Theorem \ref{thm: the unbreakable graphic matroids} implies that $\si(H_1/\cl(C))$ is a cycle on at least four vertices. This implies that $\lvert C'\rvert=2$. By replacing $\cl(C)$ with $\cl(C')$, we see that $\lvert C\rvert=2$ and $\si(H_1/\cl(C'))$ is a cycle on at least four vertices. 
    Since $\si(H_1/\cl(C))$ is a cycle and $C$ is a 2-cycle, either $\si(H_1)$ is a cycle or each edge of $C$ is in a triangle of $H_1$.
    Since $C$ and $C'$ are vertex-disjoint and $\si(H_1/\cl(C'))$ is a cycle on at least four vertices, we see that no edge of $C$ is in a triangle of $H_1$.
    Therefore $\si(H_1)$ is a cycle.
    But Proposition \ref{lem: nondegen vertex bounds} implies that $\lvert V(H_1)\rvert<5$, a contradiction. We conclude that $N_1$ is not unbreakable, a contradiction.
    \end{proof}

Theorem \ref{thm: non-degen unbreakability} has the significant consequence that the only simple and connected but not 3-connected unbreakable quasi-graphic matroids are either lifted-graphic or frame matroids, and we will consider these cases separately.
We begin with the following lemma about some quasi-graphic matroids with free elements. 
Note that we make no assumptions on $\cL$ and $\cF$.

\begin{lemma}\label{lem: complete graphs dont have free elements}
    Suppose $M=M\qbg$ where $\lvert V(G)\rvert>6$ and $\si(G)$ is obtained from a complete graph by deleting the edges of a path of length at most two. Then $M$ does not have a free element.
\end{lemma}
\begin{proof}
    We will use the language ``nearly complete" to refer to the assumed structure of $G$. The statement is clearly true if all cycles are balanced, so we may assume that $G$ has an unbalanced cycle, implying $r(M)=\lvert V(G)\rvert$. Suppose by contradiction that $M$ has a free element $p$. Since $p$ is a free element, it is not a coloop, so there is an unbalanced cycle that avoids the edge $p$. Note as well that because $p$ is a free element, every cycle using $p$ is unbalanced since a balanced cycle cannot be spanning by Proposition \ref{prop: rank of X}. Let $C$ be a shortest unbalanced cycle avoiding $p$. We claim that $\lvert C\rvert\leq3$. To see this, suppose $C$ has at least four vertices. Because $G$ is nearly complete, we can find an edge $f$ with both endpoints in $V(C)$ that is not parallel with any edge of $C$. Thus $G[C]$ contains two cycles $C_1$ and $C_2$ which are strictly smaller than $C$. One of these cycles must be unbalanced, otherwise $C$ is balanced by the theta property. But this contradicts our choice of $C$. So $\lvert C\rvert\leq3$.

    We now consider the edge $p$. First, suppose that $p$ is a loop of $G$ at vertex $v$. Since $p$ is a free element, this loop is unbalanced. Because $G$ is nearly complete, we may find a shortest (possibly empty) path $P$ from $v$ to a vertex of $C$. Note that $P$ has length at most 3, and is empty when $v$ is a vertex of $C$. Then $C\cup P\cup p$ contains a non-spanning circuit of $M$. To see this, note that if $v$ is a vertex of $C$, this circuit is a tight handcuff with both cycles unbalanced. If $v$ is not a vertex of $C$, then because $G$ has proper tripartition $(\cB,\cL,\cF)$, the cycle $C$ and the loop $p$ share the same type, and so this circuit is either a loose handcuff with both cycles in $\cF$ or a bracelet with both cycles in $\cL$. In every case this contradicts that $p$ is a free element of $M$.

    Now we may assume that $p$ is a non-loop edge of $G$ with ends $u$ and $v$. Note that because $\lvert C\rvert\leq3$ and $\lvert V(G)\rvert\geq7$, there are two vertices $x$ and $y$ that are not vertices of $C$ or incident with $p$. Since $G$ is nearly complete, we must have that both $u$ and $v$ are adjacent to one of these vertices, without loss of generality say $x$. This forms a 3-cycle using $p$, say $C'$, which must be unbalanced. Moreover, because this cycle is vertex-disjoint from $C$ and $G$ has a proper tripartition, they share the same type. Because $G$ is nearly complete, there is a vertex of $C$ that is adjacent with a vertex of $C'$ via an edge $e$ that is not an edge of either $C$ or $C'$. Then $C\cup C'\cup e$ contains a non-spanning circuit of $M$. When $C$ and $C'$ belong to $\cF$, this circuit is a loose handcuff, and when $C$ and $C'$ belong to $\cL$, this circuit is a bracelet. But this contradicts that $p$ is a free element.
    \end{proof}

We recall that any connected but not 3-connected matroid may be obtained from a sequence of 2-sums of proper minors such that each is 3-connected \cite[Corollary 8.3.4]{Oxley2011}. We may now prove Theorem \ref{thm: connected not 3-connected}, which we restate for convenience.

\begin{theorem}
    Let $M$ be a simple frame or lifted-graphic matroid that is connected but not 3-connected. Then $M$ is unbreakable if and only if $M$ is obtained via 2-sums along free elements from unbreakable frame or lifted-graphic matroids such that:
    \begin{enumerate}[$(i)$]
        \item If $M$ is a frame matroid, the rank of each summand is at most 6.
        \item If $M$ is lifted-graphic, each summand with rank greater than 6 is on a graph which simplifies to a cycle.
        \item If any summand is also graphic, it is on a graph which is a cycle on exactly 3 vertices.
    \end{enumerate}
\end{theorem}
\begin{proof}
    Suppose $M$ is unbreakable. Since $M$ is connected but not 3-connected, it is obtained via 2-sums with a sequence of 3-connected matroids, each of which is a proper minor of $M$ \cite[Corollary 8.3.4]{Oxley2011}. Since $M$ is either a frame matroid or a lifted-graphic matroid, and both of these classes are minor-closed, we may assume that either each matroid in this 2-sum is a 3-connected frame matroid, or each matroid is a 3-connected lifted-graphic matroid. 
    
    Let $M_1,\ldots,M_k$ be a minimal sequence of 3-connected proper minors of $M$ such that,
    $$
    M\cong (((M_1\oplus_{p_1}M_2)\oplus_{p_2}M_3)\oplus_{p_3}\ldots)\oplus_{p_{k-1}}M_k.
    $$
    The minimality of $k$ implies that each 2-sum in this sequence is non-trivial. Let $N_k=M$ and define $N_i$ by
    $$
    N_i=(((M_1\oplus_{p_1}M_2)\oplus_{p_2}M_3)\oplus_{p_3}\ldots)\oplus_{p_{i-1}}M_i.
    $$
    We first show that $N_i$ is simple, connected, and unbreakable for each $i\in\{1,\ldots,k\}$. We proceed similarly to \ref{sub: N_i decomp} by induction on $k-i$. By assumption, $M=N_k$ is simple, connected, and unbreakable, and we have that $N_i$ is connected for each $i$ by Proposition \cite[7.1.22]{Oxley2011}. Now suppose that $N_i$ is simple and unbreakable, and consider $N_{i-1}$. We have that $N_i=N_{i-1}\oplus_{p_{i-1}}M_{i}$. Since $N_i$ is simple and each 2-sum is non-trivial, $r(N_{i-1})\geq2$ and $r(M_i)\geq2$. Since $N_i$ is unbreakable, Lemma \ref{lem: 2-sum is unbreakable iff free element} implies that the basepoint $p_{i-1}$ is a free element of both $N_{i-1}$ and $M_i$. Since these matroids are connected (and therefore loopless), Lemma \ref{lem: free element implies unbreakable} implies that they are both unbreakable. To show that $N_{i-1}$ is simple, we have the straightforward observation that $N_{i-1}$ may not have any parallel pairs that are disjoint from $p_{i-1}$ since $N_i$ is simple. Furthermore, since $r(N_{i-1})\geq2$, any parallel pairs containing $p_{i-1}$ are not spanning in $N_{i-1}$, contradicting that $p_{i-1}$ is a free element. Thus $N_{i-1}$ has no parallel pairs, and, since it is connected, $N_{i-1}$ has no loops. Thus $N_{i-1}$ is simple, as desired.

    A particularly relevant observation from the previous argument is that, because each 2-sum is non-trivial, $r(N_i) \ge 2$ and $r(M_i)\geq2$ for each $i$. Because $N_i$ is unbreakable, Lemma \ref{lem: 2-sum is unbreakable iff free element} implies that each basepoint $p_i$ is a free element of $M_{i+1}$ (noting that $p_1$ is also a free element of $M_1$). Since $M_i$ is 3-connected for each $i$, it is loopless, so Lemma \ref{lem: free element implies unbreakable} implies that $M_i$ is unbreakable for each $i$. We now have that each matroid $M_i$ is 3-connected and unbreakable, and may apply Theorem \ref{thm: main}. We recall that for each matroid $M_i$, we may choose a biased graph $(G_i,\cB_i)$ representing $M_i$ such that $G_i$ has no component with fewer than two vertices.

    First suppose that $M$ (and therefore each $M_i$) is a frame matroid. Suppose that $r(M_i)\geq7$ for some $i\in\{1,\ldots,k\}$. Theorem \ref{thm: main} together with Lemma \ref{lem: complete graphs dont have free elements} implies that $\lvert V(G_i)\rvert\leq6$, so $r(M_i)\leq6$, a contradiction. Therefore $(i)$ holds.

    Next suppose that $M$ (and therefore each $M_i$) is lifted-graphic. If $r(M_i)\geq7$ for some $i\in\{1,\ldots,k\}$, then Theorem \ref{thm: main} and Lemma \ref{lem: complete graphs dont have free elements} imply that $\si(G_i)$ is a cycle, so $(ii)$ holds.

    Finally, if any of the matroids $M_i$ is graphic, then Lemma \ref{lem: link-sum graphic matroids are cycles} implies that the graph of $M_i$ is a cycle on at least 3 vertices.  Since $M_i$ is 3-connected, the fact that the graphs of cycles are not 3-connected implies that $\lvert V(G_i)\rvert=3$ \cite[Proposition 8.1.9]{Oxley2011}.

    Conversely, suppose that a simple quasi-graphic matroid $M$ is obtained via 2-sums of a sequence of unbreakable matroids $M_1,\ldots,M_k$, with corresponding free element basepoints $p_1,\ldots,p_{k-1}$. We may assume that $k$ is minimal, so that each of these 2-sums is non-trivial. We wish to show that $r(M_i)\geq2$ and apply Lemma \ref{lem: 2-sum is unbreakable iff free element}. Since each matroid $M_i$ is unbreakable, it is connected, so $M_i$ has no loops. Since the basepoints $p_i$ are free elements (and therefore not coloops), $M_i$ has at least two elements. If $M_i$ consists of a parallel pair containing $p_i$, this contradicts that the 2-sums are non-trivial. If $M_i$ consists of a parallel class on $m$ elements with $m \ge 3$, then, since $M$ is simple, one element of this parallel class must be a basepoint $p_j$ of the $2$-sum of $((M_1\oplus_{p_1}M_2)\oplus_{p_2}\ldots)\oplus_{p_{j}}M_{j+1}$ with $M_{j+1}$.
    But then $p_j$ is not a free element of $((M_1\oplus_{p_1}M_2)\oplus_{p_2}\ldots)\oplus_{p_{j}}M_{j+1}$ because this matroid has rank at least two, a contradiction.
    Thus $M_i$ contains an element which is not parallel with $p_i$, implying that $r(M_i)\geq2$ for each $i\in\{1,\ldots,k\}$. Lemma \ref{lem: 2-sum is unbreakable iff free element} then implies that each 2-sum in this sequence is unbreakable, and it follows that $M$ is unbreakable.
\end{proof}


\bibliographystyle{abbrv}
\bibliography{references.bib}

@article {Oxley-Pfeil2022,
    AUTHOR = {Oxley, James and Pfeil, Simon},
     TITLE = {Unbreakable matroids},
   JOURNAL = {Adv. in Appl. Math.},
  FJOURNAL = {Advances in Applied Mathematics},
    VOLUME = {141},
      YEAR = {2022},
     PAGES = {Paper No. 102404, 7},
      ISSN = {0196-8858},
   MRCLASS = {05B35},
  MRNUMBER = {4456793},
MRREVIEWER = {Maruti M. Shikare},
       DOI = {10.1016/j.aam.2022.102404},
       URL = {https://doi.org/10.1016/j.aam.2022.102404},
}

@article {Fife-Mayhew-Oxley-Semple2020,
    AUTHOR = {Fife, Tara and Mayhew, Dillon and Oxley, James and Semple,
              Charles},
     TITLE = {The unbreakable frame matroids},
   JOURNAL = {SIAM J. Discrete Math.},
  FJOURNAL = {SIAM Journal on Discrete Mathematics},
    VOLUME = {34},
      YEAR = {2020},
    NUMBER = {3},
     PAGES = {1522--1537},
      ISSN = {0895-4801},
   MRCLASS = {05B35 (05C40)},
  MRNUMBER = {4121881},
MRREVIEWER = {B. N. Waphare},
       DOI = {10.1137/19M126428X},
       URL = {https://doi.org/10.1137/19M126428X},
}

@book {Oxley2011,
    AUTHOR = {Oxley, James},
     TITLE = {Matroid theory},
    SERIES = {Oxford Graduate Texts in Mathematics},
    VOLUME = {21},
   EDITION = {Second},
 PUBLISHER = {Oxford University Press, Oxford},
      YEAR = {2011},
     PAGES = {xiv+684},
      ISBN = {978-0-19-960339-8},
   MRCLASS = {05-01 (05B35 90C27)},
  MRNUMBER = {2849819},
MRREVIEWER = {Maruti M. Shikare},
       DOI = {10.1093/acprof:oso/9780198566946.001.0001},
       URL = {https://doi.org/10.1093/acprof:oso/9780198566946.001.0001},
}

@article {Zaslavsky1989,
    AUTHOR = {Zaslavsky, Thomas},
     TITLE = {Biased graphs. {I}. {B}ias, balance, and gains},
   JOURNAL = {J. Combin. Theory Ser. B},
  FJOURNAL = {Journal of Combinatorial Theory. Series B},
    VOLUME = {47},
      YEAR = {1989},
    NUMBER = {1},
     PAGES = {32--52},
      ISSN = {0095-8956},
   MRCLASS = {05C99 (05B35)},
  MRNUMBER = {1007712},
MRREVIEWER = {J. M. S. Sim\~{o}es-Pereira},
       DOI = {10.1016/0095-8956(89)90063-4},
       URL = {https://doi.org/10.1016/0095-8956(89)90063-4},
}

@article {Zaslavsky1991,
    AUTHOR = {Zaslavsky, Thomas},
     TITLE = {Biased graphs. {II}. {T}he three matroids},
   JOURNAL = {J. Combin. Theory Ser. B},
  FJOURNAL = {Journal of Combinatorial Theory. Series B},
    VOLUME = {51},
      YEAR = {1991},
    NUMBER = {1},
     PAGES = {46--72},
      ISSN = {0095-8956},
   MRCLASS = {05B35 (05C99)},
  MRNUMBER = {1088626},
MRREVIEWER = {J. M. S. Sim\~{o}es-Pereira},
       DOI = {10.1016/0095-8956(91)90005-5},
       URL = {https://doi.org/10.1016/0095-8956(91)90005-5},
}

@article {Bowler-Funk-Slilaty20,
    AUTHOR = {Bowler, Nathan and Funk, Daryl and Slilaty, Daniel},
     TITLE = {Describing quasi-graphic matroids},
   JOURNAL = {European J. Combin.},
  FJOURNAL = {European Journal of Combinatorics},
    VOLUME = {85},
      YEAR = {2020},
     PAGES = {103062, 26},
      ISSN = {0195-6698},
   MRCLASS = {05B35},
  MRNUMBER = {4037634},
MRREVIEWER = {Brigitte Servatius},
       DOI = {10.1016/j.ejc.2019.103062},
       URL = {https://doi.org/10.1016/j.ejc.2019.103062},
}

@article {Bowler-Funk-Slilaty24-corrigendum,
    AUTHOR = {Bowler, Nathan and Funk, Daryl and Slilaty, Daniel},
     TITLE = {Corrigendum to ``{D}escribing quasi-graphic matroids''
              [{E}uropean {J}. {C}ombin. 85 (2020) 103062]},
   JOURNAL = {European J. Combin.},
  FJOURNAL = {European Journal of Combinatorics},
    VOLUME = {122},
      YEAR = {2024},
     PAGES = {Paper No. 104004, 4},
      ISSN = {0195-6698},
   MRCLASS = {05B35},
  MRNUMBER = {4789251},
       DOI = {10.1016/j.ejc.2024.104004},
       URL = {https://doi.org/10.1016/j.ejc.2024.104004},
}

@article {Geelen-Gerards-Whittle18,
    AUTHOR = {Geelen, Jim and Gerards, Bert and Whittle, Geoff},
     TITLE = {Quasi-graphic matroids},
   JOURNAL = {J. Graph Theory},
  FJOURNAL = {Journal of Graph Theory},
    VOLUME = {87},
      YEAR = {2018},
    NUMBER = {2},
     PAGES = {253--264},
      ISSN = {0364-9024},
   MRCLASS = {05B35 (05C83)},
  MRNUMBER = {3742182},
MRREVIEWER = {Maruti M. Shikare},
       DOI = {10.1002/jgt.22177},
       URL = {https://doi.org/10.1002/jgt.22177},
}

@article {Drummond-Fife-Grace-Oxley-2020,
    AUTHOR = {Drummond, George and Fife, Tara and Grace, Kevin and Oxley,
              James},
     TITLE = {Circuit-difference matroids},
   JOURNAL = {Electron. J. Combin.},
  FJOURNAL = {Electronic Journal of Combinatorics},
    VOLUME = {27},
      YEAR = {2020},
    NUMBER = {3},
     PAGES = {Paper No. 3.11, 11},
   MRCLASS = {05B35},
  MRNUMBER = {4245124},
MRREVIEWER = {Hery Randriamaro},
       DOI = {10.37236/9314},
       URL = {https://doi.org/10.37236/9314},
}

@article {Cho-Oxley-Wang-2026,
    AUTHOR = {Cho, Christine and Oxley, James and Wang, Suijie},
     TITLE = {The symmetric strong circuit elimination property},
   JOURNAL = {Adv. in Appl. Math.},
  FJOURNAL = {Advances in Applied Mathematics},
    VOLUME = {173},
      YEAR = {2026},
    NUMBER = {part A},
     PAGES = {Paper No. 102983},
      ISSN = {0196-8858},
   MRCLASS = {05B35},
  MRNUMBER = {4969865},
       DOI = {10.1016/j.aam.2025.102983},
       URL = {https://doi.org/10.1016/j.aam.2025.102983},
}

@article {Geelen-Nelson-2010,
    AUTHOR = {Geelen, Jim and Nelson, Peter},
     TITLE = {The number of points in a matroid with no {$n$}-point line as
              a minor},
   JOURNAL = {J. Combin. Theory Ser. B},
  FJOURNAL = {Journal of Combinatorial Theory. Series B},
    VOLUME = {100},
      YEAR = {2010},
    NUMBER = {6},
     PAGES = {625--630},
      ISSN = {0095-8956},
   MRCLASS = {05B35 (05C83 51A45)},
  MRNUMBER = {2718682},
MRREVIEWER = {Stefan H. M. van Zwam},
       DOI = {10.1016/j.jctb.2010.06.001},
       URL = {https://doi.org/10.1016/j.jctb.2010.06.001},
}

@article {Nelson-2013,
    AUTHOR = {Nelson, Peter},
     TITLE = {Growth rate functions of dense classes of representable
              matroids},
   JOURNAL = {J. Combin. Theory Ser. B},
  FJOURNAL = {Journal of Combinatorial Theory. Series B},
    VOLUME = {103},
      YEAR = {2013},
    NUMBER = {1},
     PAGES = {75--92},
      ISSN = {0095-8956},
   MRCLASS = {05B35 (05B25)},
  MRNUMBER = {2995720},
MRREVIEWER = {Stefan H. M. van Zwam},
       DOI = {10.1016/j.jctb.2012.09.002},
       URL = {https://doi.org/10.1016/j.jctb.2012.09.002},
}

@article {Chen-Geelen18,
    AUTHOR = {Chen, Rong and Geelen, Jim},
     TITLE = {Infinitely many excluded minors for frame matroids and for
              lifted-graphic matroids},
   JOURNAL = {J. Combin. Theory Ser. B},
  FJOURNAL = {Journal of Combinatorial Theory. Series B},
    VOLUME = {133},
      YEAR = {2018},
     PAGES = {46--53},
      ISSN = {0095-8956},
   MRCLASS = {05B35},
  MRNUMBER = {3856704},
MRREVIEWER = {Joseph E. Bonin},
       DOI = {10.1016/j.jctb.2018.04.003},
       URL = {https://doi.org/10.1016/j.jctb.2018.04.003},
}

@article {Chen-Whittle18,
    AUTHOR = {Chen, Rong and Whittle, Geoff},
     TITLE = {On recognizing frame and lifted-graphic matroids},
   JOURNAL = {J. Graph Theory},
  FJOURNAL = {Journal of Graph Theory},
    VOLUME = {87},
      YEAR = {2018},
    NUMBER = {1},
     PAGES = {72--76},
      ISSN = {0364-9024},
   MRCLASS = {05B35},
  MRNUMBER = {3729836},
MRREVIEWER = {Haidong Wu},
       DOI = {10.1002/jgt.22141},
       URL = {https://doi.org/10.1002/jgt.22141},
}

@article {Funk-Pivotto-Slilaty22,
    AUTHOR = {Funk, Daryl and Pivotto, Irene and Slilaty, Daniel},
     TITLE = {Matrix representations of frame and lifted-graphic matroids
              correspond to gain functions},
   JOURNAL = {J. Combin. Theory Ser. B},
  FJOURNAL = {Journal of Combinatorial Theory. Series B},
    VOLUME = {155},
      YEAR = {2022},
     PAGES = {202--255},
      ISSN = {0095-8956},
   MRCLASS = {05B35 (05C22)},
  MRNUMBER = {4392273},
MRREVIEWER = {Haidong Wu},
       DOI = {10.1016/j.jctb.2022.02.007},
       URL = {https://doi.org/10.1016/j.jctb.2022.02.007},
}

\end{document}